\documentclass[oneside,english]{amsart}
\usepackage[T1]{fontenc}
\usepackage[latin9]{inputenc}
\usepackage{geometry}
\usepackage{textcomp}
\usepackage{mathrsfs}
\usepackage{amstext}
\usepackage{amsthm}
\usepackage{amssymb}

\makeatletter
\numberwithin{equation}{section}
\numberwithin{figure}{section}
\theoremstyle{plain}
\newtheorem{thm}{\protect\theoremname}[section]
\theoremstyle{plain}
\newtheorem{conjecture}[thm]{\protect\conjecturename}
\theoremstyle{plain}
\newtheorem{cor}[thm]{\protect\corollaryname}
\theoremstyle{remark}
\newtheorem{rem}[thm]{\protect\remarkname}
\theoremstyle{plain}
\newtheorem{lem}[thm]{\protect\lemmaname}
\theoremstyle{plain}
\newtheorem{prop}[thm]{\protect\propositionname}

\def\makebbb#1{
    \expandafter\gdef\csname#1\endcsname{
        \ensuremath{\Bbb{#1}}}
}\makebbb{R}\makebbb{N}\makebbb{Z}\makebbb{C}\makebbb{H}\makebbb{E}\makebbb{H}\makebbb{P}\makebbb{B}\makebbb{Q}\makebbb{E}\makebbb{E}
\makebbb{H}
\makebbb{F}
\usepackage{babel}

\usepackage{babel}

\makeatother

\usepackage{babel}

\makeatother

\usepackage{babel}
\providecommand{\conjecturename}{Conjecture}
\providecommand{\corollaryname}{Corollary}
\providecommand{\lemmaname}{Lemma}
\providecommand{\propositionname}{Proposition}
\providecommand{\remarkname}{Remark}
\providecommand{\theoremname}{Theorem}

\begin{document}
\title{Sharp bounds on the height of K-semistable Fano varieties II, the
log case }
\author{Robert J. Berman, Rolf Andreasson}
\begin{abstract}
In our previous work we conjectured - inspired by an algebro-geometric
result of Fujita - that the height of an arithmetic Fano variety $\mathcal{X}$
of relative dimension $n$ is maximal when $\mathcal{X}$ is the projective
space $\P_{\Z}^{n}$ over the integers, endowed with the Fubini-Study
metric, if the corresponding complex Fano variety is K-semistable.
In this work the conjecture is settled for diagonal hypersurfaces
in $\P_{\Z}^{n+1}.$ The proof is based on a logarithmic extension
of our previous conjecture, of independent interest, which is established
for toric log Fano varieties of relative dimension at most three,
hyperplane arrangements on $\P_{\Z}^{n},$ as well as for general
arithmetic orbifold Fano surfaces. 
\end{abstract}

\maketitle

\section{Introduction}

This is a sequel to \cite{a-b}, where a conjectural arithmetic analog
of Fujita's sharp bound for the degree (volume) of K-semistable Fano
varieties over $\C$ was proposed, concerning arithmetic Fano varieties
$\mathcal{X}.$ The case when $\mathcal{X}$ is the canonical integral
model of a toric Fano variety $X$ was settled in \cite{a-b}, when
the relative dimension $n$ is at most six (the extension to any $n$
is conditioned on a conjectural gap hypothesis for the algebro-geometric
degree). Here we will, in particular, show that the conjecture introduced
in \cite{a-b} holds for any diagonal Fano hypersurface $\mathcal{X}$
in $\P_{\Z}^{n+1}$ (see Section \ref{subsec:Application-to-diagonal}
below). The proof is based on the following extension of the conjecture
in \cite{a-b} to the logarithmic setting, which is the main focus
of the present work:
\begin{conjecture}
\label{conj:height log intro}Let $(\mathcal{X},\mathcal{D})$ be
an arithmetic log Fano variety. Then the following inequality of arithmetic
intersection numbers holds for any volume-normalized continuous metric
on $-(K_{X}+\Delta)$ with positive curvature current if $(X,\Delta)$
is K-semistable:
\[
(\overline{-\mathcal{K}_{(\mathcal{X},\mathcal{D})}})^{n+1}\leq(\overline{-\mathcal{K}_{\P_{\Z}^{n}}})^{n+1},
\]
where $-K_{\P_{\C}^{n}}$ is endowed with the volume-normalized Fubini-Study
metric. Moreover, if $\mathcal{X}$ is normal equality holds if and
only if $(\mathcal{X},\mathcal{D})=(\P_{\Z}^{n},0)$ and the metric
is Kähler-Einstein, i.e. coincides with the Fubini-Study metric up
to the action of an automorphism of $\P_{\C}^{n}.$
\end{conjecture}

By definition, an \emph{arithmetic log Fano variety} $(\mathcal{X},\mathcal{D})$
is a projective flat scheme $\mathcal{X}$ over $\Z$ together with
an effective $\Q-$divisor $\mathcal{D}$ on $\mathcal{X}$ such that
\[
-\mathcal{K}_{(\mathcal{X},\mathcal{D})}:=-(\mathcal{K}_{\mathcal{X}}+\mathcal{D})
\]
 defines a relatively ample $\Q-$line bundle, where $\mathcal{K}_{\mathcal{X}}$
denotes the relative canonical divisor on $\mathcal{X}.$ We also
assume that the corresponding complex variety $X$ is normal and thus
defines a complex Fano variety. Following standard procedure we denote
by $\overline{\mathcal{L}}$ a metrized line bundle, i.e. a line bundle
$\mathcal{L}$ on $\mathcal{X}$ endowed with an Hermitian metric
over the complex points $X$ of $\mathcal{X}.$ Arithmetic intersection
numbers of metrized line bundles were introduced by Gillet-Soulé in
the context of Arakelov geometry \cite{b-g-s}. The top arithmetic
intersection number of $\overline{\mathcal{L}}$ is called the\emph{
height} of $\overline{\mathcal{L}}.$ The height of $\overline{-\mathcal{K}_{\P_{\Z}^{n}}}$
with respect to the volume-normalized Fubini-Study metric, appearing
in the previous conjecture, is explicitly given by the following formula
\cite[Lemma 3.6]{ab}, which, essentially, goes back to \cite[ §5.4]{g-s}:
\begin{equation}
(\overline{-\mathcal{K}_{\P_{\Z}^{n}}})^{n+1}=\frac{1}{2}(n+1)^{n+1}\left((n+1)\sum_{k=1}^{n}k^{-1}-n+\log(\frac{\pi^{n}}{n!})\right).\label{eq:expl formul on p n}
\end{equation}
As for the notion of K-stability it originally appeared in the context
of the Yau-Tian-Donaldson conjecture for Fano manifolds $X$ (see
the survey \cite{x} for recent developments, including connections
to moduli spaces and the minimal model program in birational geometry).
By \cite{li1} and \cite[Thm 1.6]{l-x-z} a log Fano variety $(X,\Delta)$
over $\C$ is\emph{ K-polystable} (which is a slightly stronger condition
than K-semistability) if and only if it admits a \emph{log Kähler-Einstein
metric,} i.e. a locally bounded metric on $X$ whose curvature current
$\omega$ induces a Kähler metric with constant Ricci curvature on
the complement of $\Delta$ in the regular locus of $X.$ After volume-normalization
such a metric maximizes the height $(\overline{-\mathcal{K}_{(\mathcal{X},\mathcal{D})}})^{n+1}$
among all volume-normalized locally bounded metrics on $-(K_{X}+\Delta)$
with positive curvature (as shown precisely as in the case that $\mathcal{D}=0$
considered in \cite[Section 2.3]{a-b}).

The K-semistability of $(X,\Delta)$ implies that $(X,\Delta)$ is
Kawamata Log Terminal (klt) in the usual sense of birational geometry
(see Remark \ref{rem:klt}). An important class of klt log Fano varieties
$(X,\Delta)$ is provided by (smooth)\emph{ Fano orbifolds}, where
the coefficients of $\Delta$ are of the form $(1-1/m_{i})$ for positive
integers $m_{i}.$ Diophantine aspects of Fano orbifolds have recently
been explored in a number of works, building on Campana's program
\cite{ca} and its developments by Abramovich \cite{ab} (see \cite{tani}
for a very recent survey). In particular, a logarithmic generalization
of the Manin-Peyre conjecture for the density of rational points of
bounded height on Fano varieties is proposed in \cite{p-}, which,
for example, is addressed for log Fano hyperplane arrangements and
toric varieties in \cite{b-y} and \cite{p-s}, respectively. See
\cite{ber} for relations between height bounds, K-stability and the
Manin-Peyre conjecture. 

\subsection{Main results}

\subsubsection{Toric log Fano varieties}

We first consider the case when $(\mathcal{X},\mathcal{D})$ is the
canonical integral model of a toric log Fano variety $(X,\Delta)$
(see \cite[Section 2]{ma} and \cite[Def 3.5.6]{b-g-p-s}). One advantage
of the logarithmic setup is that on any given toric Fano variety $X$
there exist an infinite number of toric $\Q-$divisors $D$ such that
$-(K_{X}+\Delta)$ is a K-semistable log Fano variety. Building on
\cite{ab}, where the case when $D=0$ was considered, we show 
\begin{thm}
\label{thm:main log toric intro}Let $(\mathcal{X},\mathcal{D})$
be the canonical integral model of a K-semistable toric log Fano variety
$(X,\Delta).$ Conjecture \ref{conj:height log intro} holds for $(\mathcal{X},\mathcal{D})$
under anyone of the following conditions:
\begin{itemize}
\item $n\leq3$ and $X$ is $\Q-$factorial (equivalently, $X$ has at worst
abelian quotient singularities) 
\item $X$ is not Gorenstein or has some abelian quotient singularity
\end{itemize}
\end{thm}

The starting point of the proof is the bound 
\begin{equation}
\frac{\overline{(-\mathcal{K}_{(\mathcal{X},\mathcal{D})}})^{n+1}}{(n+1)!}\leq\frac{1}{2}\mathrm{vol}(X,\Delta)\log\left(\frac{(2\pi^{2})^{n}}{\mathrm{vol}(X,\Delta)}\right)\,\,\,\,\mathrm{vol}(X,\Delta):=\frac{-(K_{X}+\Delta)^{n}}{n!},\label{eq:ineq in pf thm log toric intro}
\end{equation}
shown precisely as in the case when $\Delta=0,$ considered in \cite{ab}.
For $X=\P^{n}$ the previous theorem is verified by an explicit calculation.
In the remaining case, $X\neq\P^{n},$ the bound in Conjecture \ref{conj:height log intro}
follows, just as in \cite{ab}, from combining the bound \ref{eq:ineq in pf thm log toric intro}
with the following logarithmic analog of the ``gap hypothesis''
introduced in \cite{ab}: 
\begin{equation}
\text{vol}(X,\Delta)\leq\text{vol}(\P^{n-1}\times\P^{1})\label{eq:log gap intro}
\end{equation}
 for any K-semistable $n-$dimensional Fano variety $(X,\Delta)$
such that $X\neq\P^{n}.$ In the case that $X$ is singular the logarithmic
gap hypothesis does hold in any dimension, just as in \cite{ab}.
In the non-singular case there is, for any dimension, only a finite
number of toric Fano varieties $X.$ For $n\leq6$ these appear in
the database \cite{ob}, which, as observed in \cite{ab}, settles
the gap hypothesis for $n\leq6,$ when $\Delta=0.$ However, in the
present case there is for any given toric variety $X$ an infinite
number of toric divisors $\Delta$ on $X$ such that $(X,\Delta)$
is a K-semistable Fano variety. In order to establish the logarithmic
gap-hypothesis \ref{eq:log gap intro} we thus introduce the following
invariant of a Fano manifold $X:$ 
\[
S(X):=\sup_{\Delta}\left\{ \mathrm{vol}(X,\Delta):\,(X,\Delta)\,\text{K-semistable\,log Fano}\right\} 
\]
and show, by solving the corresponding optimization problem, that
$S(X)\leq\text{vol}(\P^{n-1}\times\P^{1})$ when $X\neq\P^{n}$ and
$n\leq3.$ 

The invariant $S(X)$ and the corresponding maximizers $\Delta$ appear
to be of independent interest in Kähler geometry. This is illustrated
by some examples in Section \ref{subsec:The-logarithmic-gap}, where
we make contact with a rigidity property of the corresponding log
Kähler-Einstein metric, first exhibited in \cite{r-z}. 

\subsubsection{Hyperplane arrangements}

We next turn to the case when $\mathcal{X}$ is the projective space
over the integers, $\mathcal{X}=\P_{\Z}^{n}$ and $\mathcal{D}$ is
a hyperplane arrangement, i.e. its irreducible components are hyperplanes. 
\begin{thm}
\label{thm:hyperplane intro}Conjecture \ref{conj:height log intro}
holds when $\mathcal{X}=\P_{\Z}^{n}$ and $\mathcal{D}$ is a hyperplane
arrangement with simple normal crossings.
\end{thm}

The proof employs a convexity argument to reduce the problem to the
case when $\mathcal{D}$ is toric, which is covered by Theorem \ref{thm:main log toric}.
The argument leverages the explicit characterization of K-semistable
hyperplane arrangements established in \cite{fu2} and yields the
following explicit bound: 
\begin{equation}
\frac{(\overline{-\mathcal{K}_{(\mathcal{X},\mathcal{D})}})^{n+1}}{(n+1)!}\leq\frac{1}{2}\text{vol}(X,\Delta)\log\left(\frac{(n+1)^{n}e^{2a_{n}}}{(n+1)!\text{vol}(X,\Delta)}\right),\,\,\,a_{n}=\frac{(\overline{-\mathcal{K}_{\P_{\Z}^{n}}})^{n+1}}{\left(n+1\right)^{n+1}}\label{eq:estimate by toric intro}
\end{equation}
with equality iff $\mathcal{D}$ is toric.

\subsubsection{\label{subsec:Application-to-diagonal}Application to diagonal hypersurfaces}

Given a positive integer $d$ and integers $a_{i},$ consider the
diagonal hypersurface $\mathcal{X}_{a}$ of degree $d$ in $\P_{\Z}^{n+1}$
cut out by the homogeneous polynomial

\[
\sum_{i=0}^{n+1}a_{i}x_{0}^{d},
\]
 The corresponding complex variety $X_{a}$ is Fano if and only if
$d\leq(n+1)$ and is always K-polystable (and, in particular, K-semistable);
see, for example, \cite{zhu} for an algebraic proof. Using the results
stated in the previous two sections we will establish Conjecture \ref{conj:height log intro}
for $\mathcal{X}_{a},$ endowed with the trivial divisor $0:$ 
\begin{thm}
\label{thm:diagonal hypersurface intro}Conjecture \ref{conj:height log intro}
holds for any diagonal hypersurface $\mathcal{X}_{a}$ which is Fano
(i.e. $d\leq n+1)$ when the divisor $\mathcal{D}$ is trivial. More
precisely, 
\[
(\overline{-\mathcal{K}_{\mathcal{X}_{a}}})^{n+1}\leq(\overline{-\mathcal{K}_{\P_{\Z}^{n}}})^{n+1}+(1-d)(n+2-d)^{n}\sum_{i=0}^{n+1}\log|a_{i}|.
\]
and the inequality is strict if $d\geq2.$
\end{thm}

Note that the schemes $\mathcal{X}_{a}$ are mutually non-isomorphic
over $\Z,$ for any given degree $d$ of at least two. In fact, in
general, they are not even isomorphic over $\Q.$ The proof of the
previous theorem is first reduced to the case of a Fermat hypersurface,
i.e. the case when $a_{i}=1.$ Expressing $\mathcal{X}_{a}$ as a
Galois cover of $\P_{\Z}^{n}$ the estimate \ref{eq:estimate by toric intro}
can then be applied with $\mathcal{X}=\P_{\Z}^{n}$ and $\Delta$
the corresponding branching divisor, which reduces the problem to
a simple toric case.

We recall that the Manin-Peyre conjecture has been settled for Fano
hypersurfaces over $\Q$ of sufficiently small degree \cite{pey,pey2}.
In particular, there is an extensive literature on the diagonal case,
in connection to Waring's classical problem, where the degree bound
has been  improved \cite{pi} (see also \cite{bi} and \cite[Thm 15.6]{woo}
for general number fields).

\subsubsection{\label{subsec:Arithmetic-log-surfaces}Arithmetic log surfaces }

Consider a polarized arithmetic log surface $(\mathcal{X},\mathcal{D};\mathcal{L}),$
i.e. an arithmetic log surface $(\mathcal{X},\mathcal{D})$ endowed
with a relatively ample line bundle $\mathcal{L}.$ Assume that the
complexification $\mathcal{L}\otimes\C$ is isomorphic to $-K_{(X,\Delta)}$
(where $(X,\Delta)$ denotes, as before, the complexification of $(\mathcal{X},\mathcal{D})$).
Given a continuous metric on $\mathcal{L}$ we define the \emph{arithmetic
log Mabuchi functional} of $(\mathcal{X},\mathcal{D};\overline{\mathcal{L}})$
by
\begin{equation}
\mathcal{M}_{(\mathcal{X},\mathcal{D})}(\overline{\mathcal{L}}):=\frac{1}{2}\overline{\mathcal{L}}^{2}+\overline{\mathcal{K}}_{(\mathcal{X},\mathcal{D})}\cdot\overline{\mathcal{L}},\label{eq:def of log Mab surface intro}
\end{equation}
where $\overline{\mathcal{K}}_{(\mathcal{X},\mathcal{D})}$ is endowed
with the volume-normalized metric induced by the curvature current
$\omega$ of $\overline{\mathcal{L}},$ assuming that this metric
on $\overline{\mathcal{K}}_{(\mathcal{X},\mathcal{D})}$ is continuous.
When $\mathcal{D}=0$ the functional $\mathcal{M}_{(\mathcal{X},\mathcal{D})}(\overline{\mathcal{L}})$
coincides with the arithmetic Mabuchi functional introduced in \cite{o},
up to normalization (see \cite[Section 6.4]{a-b}). For a given integral
model $\mathcal{M}_{(\mathcal{X},\mathcal{D})}(\overline{\mathcal{L}})$
is minimized on a log Kähler-Einstein metric, if such a metric exists,
and then

\[
\mathcal{M}_{(\mathcal{X},\mathcal{D})}(\overline{\mathcal{L}})=-(\overline{-\mathcal{K}_{(\mathcal{X},\mathcal{D})}})^{2}/2,
\]
 if the metric is volume-normalized and $\mathcal{L}=-\overline{\mathcal{K}}_{(\mathcal{X},\mathcal{D})}.$ 
\begin{thm}
\label{thm:Mab on log surface}Let $(\mathcal{X},\mathcal{D};\mathcal{L})$
be a polarized arithmetic log surface $(\mathcal{X},\mathcal{D};\mathcal{L})$
with $\mathcal{X}$ normal, such that the complexification $(X,\Delta)$
of $(\mathcal{X},\mathcal{D})$ is a K-semistable Fano variety and
$\mathcal{L}\otimes\C=-K_{(X,\Delta)}.$ If $\Delta$ is supported
on (at most) three points, then 

\[
\mathcal{M}_{(\mathcal{X},\mathcal{D})}(\overline{\mathcal{L}})\geq\mathcal{M}_{(\P_{\Z}^{1},0)}(\overline{-\mathcal{K}_{\P_{\Z}^{1}}})\,\,\,\left(=-1-\log\pi\right)
\]
 where $-\mathcal{K}_{\P_{\Z}^{1}}$ is endowed with the Fubini-Study
metric. Moreover, equality holds iff $(\mathcal{X},\mathcal{D})$
is isomorphic to $(\P_{\Z}^{1},0)$ and $\mathcal{L}$ is isomorphic
to $-\mathcal{K}_{\P_{\Z}^{1}},$ endowed with a metric coinciding
with the Fubini-Study metric, up to the application of an automorphism
of $\P_{\Z}^{1}$ and a scaling of the metric.
\end{thm}

\begin{cor}
\label{cor:Conjecture holds for log Fano surfaces}Conjecture \ref{conj:height log intro}
holds for arithmetic normal log Fano surfaces, if $\Delta$ is supported
on, at most, three points. In particular, the conjecture holds for
all normal arithmetic log Fano orbifold surfaces (i.e. the case of
coefficents of the form $1-1/m_{i}$ for $m_{i}\in\N$).
\end{cor}

In the setup of the previous theorem the corresponding complex variety
$X$ is always equal to $\P^{1}$ and thus $(X,\Delta)$ is a hyperplane
arrangement. Accordingly, applying Theorem \ref{thm:hyperplane intro},
the proof of Theorem \ref{thm:Mab on log surface} is reduced to showing
that the canonical integral model $(\mathcal{X}_{c},\mathcal{D}_{c};-\mathcal{K}_{(\mathcal{X}_{c},\mathcal{D}_{c})})$
of $(X,\Delta;-K_{(X,\Delta)})$ obtained by setting $\mathcal{X}_{c}=\P_{\Z}^{1}$
and taking $\mathcal{D}_{c}$ to be the Zariski closure of $\{0,1,\infty)\}$
in $\P_{\Z}^{1}$ minimizes $\mathcal{M}_{(\mathcal{X},\mathcal{D})}(\overline{\mathcal{L}})$
over all integral models $(\mathcal{X},\mathcal{D};\mathcal{L})$
of $(X,\Delta;-K_{(X,\Delta)}$), for any fixed metric on $-K_{(X,\Delta)}.$
This minimization property can be viewed as logarithmic version of
Odaka's minimization conjecture (proposed in any dimension in \cite{o}).
Our proof builds on \cite{o}, leveraging log canonical thresholds.
We also show that the minimum is uniquely attained for $(\P_{\Z}^{1},0).$
See also \cite{h-o} for very recent progress on Odaka's minimization
conjecture in another direction.

\subsection{Outlook on the case of $\P_{\Z}^{1}$ endowed with a divisor with
three components}

Consider now the special case of $(\P_{\Z}^{1},\mathcal{D}_{c}),$
where $\mathcal{D}_{c}$ is the Zariski closure of $\{0,1,\infty)\}$
in $\P_{\Z}^{1}$ and $-K_{(\P^{1},\Delta_{\Q})}$ is endowed with
the volume-normalized log Kähler-Einstein metric. In this case an
explicit formula for the canonical height $\overline{(-\mathcal{K}_{(\P_{\Z}^{1},\mathcal{D}_{c})}})^{2}$
is established in the sequel \cite{a-b2} to the present work. More
precisely, expressing 
\[
\Delta_{\Q}=\sum_{i=1}^{3}w_{i}p_{i},\,\,\,V=2-\sum_{i=1}^{3}w_{i}>0
\]
 (where $V$ is the algebraic degree of $-K_{(\P^{1},\Delta_{\Q})})$
it is shown in \cite{a-b2} that, when $(\P_{1},\Delta_{\Q})$ is
K-polystable,

\begin{equation}
\frac{(\overline{-\mathcal{K}_{(\P_{\Z}^{1},\mathcal{D}_{c})}})^{2}}{2V}=\frac{1}{2}(1+\log\pi-\log\frac{V}{2})-\frac{\gamma(0,\frac{V}{2})+\sum_{i=1}^{3}\gamma(w_{i},w_{i}+\frac{V}{2})}{V},\label{eq:height with zeta}
\end{equation}
 where $\gamma(a,b)$ may be expressed in terms of Hurwitz zeta function
$\zeta(s,x)$ and its derivative $\zeta'(s,x)$ wrt $s:$
\[
\gamma(a,b)=F(b)+F(1-b)-F(a)-F(1-a),\,\,\,F(x):=\zeta(-1,x)+\zeta'(-1,x).
\]
 Moreover, the right hand side of formula \ref{eq:height with zeta}
is shown to extend real-analytically, wrt the coefficients $w_{i},$
to the case $K_{(\P^{1},\Delta_{\Q})}>0$ (i.e. $V<0),$ as long as
$w_{i}\in]0,1].$ In this case formula \ref{eq:height with zeta}
computes $(\mathcal{K}_{(\P_{\Z}^{1},\mathcal{D}_{c})})^{2}/2V$ and
can be related to Kudla's program \cite{kud} and the Maillot-Rössler
conjectures \cite{m-r1,m-r2}, expressing the height of Shimura varieties
$(\mathcal{X},\mathcal{D}),$ wrt canonical metrics, in terms of the
Dedekind zeta function $\zeta_{\F}(s)$ of the number field $\F$
attached to $(\mathcal{X},\mathcal{D})$ and its derivative $\zeta'_{\F}(s)$
at $s=-1.$ More precisely, this connection concerns the $18$ Shimura
varieties that are isomorphic to $(\P^{1},\Delta)$ over $\F,$ for
a particular orbifold divisor $\Delta.$ In particular, combining
formula \ref{eq:height with zeta} with the recent work \cite{yu1}
on modular heights yields information about the canonical integral
models of the corresponding Shimura curves.

\subsection{Acknowledgments}

Thanks to Julia Brandes, Dennis Eriksson, Mattias Jonsson, Gerard
Freixas i Montplet and Yuji Odaka for discussions. This work was supported
by a Wallenberg Scholar grant from the Knut and Alice Wallenberg foundation.

\section{\label{sec:General-setup}General setup}

\subsection{\label{subsec:Log-Fano-varieties metrics}Log Fano varieties over
$\C$ and volume-normalized metrics on $-(K_{X}+\Delta)$ }

A \emph{log pair }$(X,\Delta)$ over $\C$ is a normal complex projective
variety $X$ together with an effective $\Q-$divisor $\Delta$ on
$X$ such that $K_{X}+\Delta$ is $\Q-$Cartier, i.e. defines a $\Q-$line
bundle, where $K_{X}$ denotes the canonical divisor on $X$ \cite{ko0}.
In the logarithmic setting this bundle plays the role of the canonical
line bundle and is thus called the\emph{ log canonical line bundle}
and is denoted by $K_{(X,\Delta)}.$ A log pair $(X,\Delta)$ is said
to be a\emph{ log Fano pair} if $\Delta$ is effective and $-(K_{X}+\Delta)>0.$
Any continuous metric $\left\Vert \cdot\right\Vert $ on $-(K_{X}+\Delta)$
induces a measure $\mu$ on $X$ in a standard fashion. Indeed, when
$X$ is regular and $\Delta=0$ this follows directly from the definition
of metrics on $-K_{X}$ (see \cite[Section 2.1.2]{a-b}). In general,
denoting by $X_{reg}$ the regular locus of $X,$ this construction
yields a measure on $X_{reg}-\text{supp}(\Delta)$ whose push-forward
to $X,$ under the inclusion map, yields a measure on $X$ (see also
\cite[Section 3.1]{bbegz} for a slightly different representation
of this measure). This measure has finite mass iff the log pair $(X,\Delta)$
is \emph{klt} in the standard sense of birational algebraic geometry
(see \cite[Section 3.1]{bbegz} and Remark \ref{rem:klt} below).
A continuous metric on $-(K_{X}+\Delta)$ will be said to be \emph{volume-normalized}
if the corresponding measure is a probability measure. 

\subsubsection{\label{subsec:Local-representations-of}Local representations of
metrics and measures}

As in \cite{ab} we will use additive notation for metrics on holomorphic
line bundles $L\rightarrow X.$ This means that we identify a continuous
Hermitian metric $\left\Vert \cdot\right\Vert $ on $L$ with a collection
of continuous local functions $\phi_{U}$ associated to a given covering
of $X$ by open subsets $U$ and trivializing holomorphic sections
$e_{U}$ of $L\rightarrow U:$
\begin{equation}
\phi_{U}:=-\log(\left\Vert e_{U}\right\Vert ^{2}).\label{eq:def of phi U}
\end{equation}
The curvature current of the metric may then, locally, be expressed
as 

\[
dd^{c}\phi_{U}:=\frac{i}{2\pi}\partial\bar{\partial}\phi_{U}
\]
Accordingly, as is customary, we will symbolically denote by $\phi$
a given continuous Hermitian metric on $L$ and by $dd^{c}\phi$ its
curvature current. We will denote by $\mathcal{C}^{0}(L)\cap\text{PSH\ensuremath{(L)}}$
the space of all continuous metrics on $L$ whose curvature current
is positive, $dd^{c}\phi\geq0$ (which means that $\phi_{U}$ is plurisubharmonic,
or psh, for short). 

Given a log Fano pair $(X,\Delta)$ the measure corresponding to a
given continuous metric $\phi$ on $-K_{(X,\Delta)}$ may be locally
on $X_{reg}$ be expressed as
\[
\mu_{\phi}=e^{-\phi_{U}}\left|s_{U}\right|^{-2}(\frac{i}{2})^{n^{2}}dz\wedge d\bar{z},\,\,\,dz:=dz_{1}\wedge\cdots\wedge dz_{n}
\]
 by taking $e_{U}=\partial/\partial z_{1}\wedge\cdots\wedge\partial/\partial z_{n}\otimes e_{\Delta}$
where $e_{\Delta}$ is a local trivialization of the $\Q-$line bundle
over $X_{reg}$ corresponding to the divisor $\Delta$ and $s_{U}e_{\Delta}$
is the (multi-valued) holomorphic section cutting out $\Delta.$ 

\subsubsection{Log Kähler-Einstein metrics}

Given a log Fano pair $(X,\Delta)$ a metric $\phi$ on $(X,\Delta)$
is said to be a\emph{ log Kähler-Einstein metric,} if $\phi$ is a
locally bounded metric and its curvature current $dd^{c}\phi$ induces
a Kähler metric with constant positive Ricci curvature on the complement
of $\Delta$ in $X_{reg}$ \cite{bbegz}. When $X$ is smooth any
log Kähler-Einstein metric is, in fact, continuous (see \cite{j-m-r,g-m}
for more general higher order regularity results).

\subsubsection{K-semistability}

We next recall the definition of K-semistability in terms of intersection
numbers (see the survey \cite{x} for more background). A \emph{test
configuration} for a log Fano pair $(X,\Delta)$ is a $\C^{*}-$equivariant
normal model $(\mathscr{X},\mathscr{L})$ for $(X,-(K_{(X,\Delta)})$
over the complex affine line $\C.$ More precisely, $\mathscr{X}$
is a normal complex variety endowed with a $\C^{*}-$action $\rho$,
a $\C^{*}-$equivariant holomorphic projection $\pi$ to $\C$ and
a relatively ample $\C^{*}-$equivariant $\Q-$line bundle $\mathscr{L}$
(endowed with a lift of $\rho):$ 
\begin{equation}
\pi:\mathcal{\mathscr{X}}\rightarrow\C,\,\,\,\,\,\mathscr{L}\rightarrow\mathscr{X},\,\,\,\,\,\,\rho:\,\,\mathscr{X}\times\C^{*}\rightarrow\mathscr{X}\label{eq:def of pi for test c}
\end{equation}
such that the fiber of $\mathscr{X}$ over $1\in\C$ is equal to $(X,-(K_{(X,\Delta)}).$
A log Fano pair $(X,\Delta)$ is said to be \emph{K-semistable} if
the Donaldson-Futaki invariants $\text{DF}_{\Delta}(\mathscr{X},\mathscr{L})$
are non-negative for any test configuration $(\mathscr{X},\mathscr{L})$
of $(X,\Delta):$ 
\[
n!\text{DF}_{\Delta}(\mathscr{X},\mathscr{L})=\frac{n}{(n+1)}\overline{\mathscr{L}}^{n+1}+\mathscr{K}_{(\mathcal{\mathscr{\overline{X}}},\mathscr{D})/\P^{1}}\cdot\mathcal{\overline{\mathscr{L}}}^{n},
\]
 where $\overline{\mathscr{L}}$ denotes the $\C^{*}-$equivariant
extension of $\mathscr{L}$ to the $\C^{*}-$equivariant compactification
$\mathscr{\overline{X}}$ of $\mathscr{X}$ over $\P^{1}$ and $\mathscr{K}_{(\mathcal{\mathscr{\overline{X}}},\mathscr{D})/\P^{1}}$
denotes the relative log canonical divisor of the pair $(\mathscr{\overline{X}}$,$\mathscr{D}$)
with $\mathscr{D}$ denoting the flat closure in $\mathscr{\overline{X}}$
of the $\C^{*}-$orbit of the divisor $\Delta.$ 
\begin{rem}
\label{rem:klt}If a log Fano variety $(X,\Delta)$ is K-semistable,
then $(X,\Delta)$ is klt \cite[Cor 9.6]{b-h-j}. When $X$ is non-singular
and $\Delta$ has simple normal crossings this means that all the
coefficients of $\Delta$ along its irreducible components are strictly
smaller than $1.$ 
\end{rem}

\subsection{Arithmetic log Fano varieties}

As explained in the book \cite{ko} the notion of log pairs can be
extended to schemes over excellent rings. Here we will consider the
case when the ring in question is $\Z.$ Henceforth, $\mathcal{X}$
will denote an \emph{arithmetic variety,} i.e. a projective flat scheme
$\mathcal{X}\rightarrow\text{Spec \ensuremath{\Z}}$ of relative dimension
$n$ such that $\mathcal{X}$ is reduced and satisfies Serre's conditions
$S_{2}$ (this is, for example, the case if $\mathcal{X}$ is normal).
We will denote by $\pi$ the corresponding structure morphism to $\text{Spec \ensuremath{\Z}, }$
\[
\pi:\,\mathcal{X}\rightarrow\text{Spec \ensuremath{\Z}.}
\]
A \emph{log pair} $(\mathcal{X},\mathcal{D})$ over $\Z$ (also called
an\emph{ arithmetic log variety}) of relative dimension $n$ is an
arithmetic variety $\mathcal{X}$ endowed with an effective $\Q-$divisor
$\mathcal{D}$ on $\mathcal{X}$ such that $\mathcal{K}_{\mathcal{X}}+\mathcal{D}$
is $\Q-$Cartier, i.e. defines a $\Q-$line bundle, where $\mathcal{K}_{\mathcal{X}}$
denotes the relative canonical divisor on $\mathcal{X}$ (see \cite[Section 1.1]{ko}).
The complexification of $(\mathcal{X},\mathcal{D})$ will be denoted
by $(X,\Delta)$ and $(\mathcal{X},\mathcal{D})$ will be called an
\emph{integral model of $(X,\Delta)$} (although, strictly speaking,
$(\mathcal{X},\mathcal{D})$ is an integral model of the corresponding
log pair over $\Q$). A log pair $(\mathcal{X},\mathcal{D})$ over
$\Z$ will be called an\emph{ arithmetic log Fano variety} if $-(\mathcal{K}_{\mathcal{X}}+\mathcal{D})$
is relatively ample and the corresponding complex variety $X$ is
normal. In particular, $(X,\Delta)$ is a log Fano variety over $\C.$ 

More generally, an arithmetic variety $\mathcal{X}$ endowed with
a relatively ample line bundle $\mathcal{L}$ will be said to be \emph{polarized. }

\subsection{Arithmetic intersection numbers and heights}

We recall some well-known facts about heights (see\cite{a-b} for
more background and references). A\emph{ metrized line bundle} $\overline{\mathcal{L}}$
is a line bundle $\mathcal{L}\rightarrow\mathcal{X}$ such that the
corresponding line bundle $L\rightarrow X$ is endowed with a metric,
that we shall denote by $\phi$ (as in Section \ref{subsec:Local-representations-of});
$\overline{\mathcal{L}}:=\left(\mathcal{L},\phi\right).$ The \emph{$\chi-$arithmetic
volume} of a polarized arithmetic variety $(\mathcal{X},\mathcal{L})$
is defined by 

\begin{equation}
\widehat{\text{vol}}_{\chi}\left(\overline{\mathcal{L}}\right):=\lim_{k\rightarrow\infty}k^{-(n+1)}\log\text{Vol}\text{\ensuremath{\left\{  s_{k}\in H^{0}(\mathcal{X},k\mathcal{L})\otimes\R:\,\,\,\sup_{X}\left\Vert s_{k}\right\Vert _{\phi}\leq1\right\} } , }\label{eq:def of xhi vol}
\end{equation}
where $H^{0}(\mathcal{X},k\mathcal{L})\otimes\R$ may be identified
with the subspace of real sections in $H^{0}(X,kL).$ More generally,
$\widehat{\text{vol}}_{\chi}\left(\overline{\mathcal{L}}\right)$
is naturally defined for $\Q-$line bundles, since it is homogeneous
with respect to tensor products of $\overline{\mathcal{L}}:$ 
\begin{equation}
\widehat{\text{vol}}_{\chi}\left(m\overline{\mathcal{L}}\right)=m^{n+1}\widehat{\text{vol}}_{\chi}\left(\overline{\mathcal{L}}\right),\,\,\,\text{if \ensuremath{m\in\Z_{+}}}\label{eq:chi vol for tensor}
\end{equation}
Moreover, $\widehat{\text{vol}}_{\chi}\left(\overline{\mathcal{L}}\right)$
is additively equivariant with respect to scalings of the metric:
\begin{equation}
\widehat{\text{vol}}_{\chi}\left(\mathcal{L},\phi+\lambda\right)=\widehat{\text{vol}}_{\chi}\left(\overline{\mathcal{L}}\right)+\frac{\lambda}{2}\text{vol\ensuremath{(L)}},\,\,\,\text{if }\lambda\in\R.\label{eq:scaling of chi vol}
\end{equation}
 If the metric on $L$ has positive curvature current (i.e. if $\phi$
is psh), then, by the arithmetic Hilbert-Samuel theorem,
\begin{equation}
\widehat{\text{vol}}_{\chi}\left(\overline{\mathcal{L}}\right)=\frac{\overline{\mathcal{L}}^{n+1}}{(n+1)!},\label{eq:vol chi as intersection}
\end{equation}
 where $\overline{\mathcal{L}}^{n+1}$ denotes the top arithmetic
intersection number in the sense of Gillet-Soulé \cite{g-s}, which,
defines the \emph{height} of $\mathcal{X}$ with respect to $\overline{\mathcal{L}}$
\cite{fa,b-g-s}. For the purpose of the present paper formula \ref{eq:vol chi as intersection}
may be taken as the definition of $\overline{\mathcal{L}}^{n+1}$
(arithmetic intersections between general $n+1$ metrized line bundles
could then be defined by polarization, i.e. using multilinearity).
Following standard practice we will use the shorthand $h_{\phi}(\mathcal{X},\mathcal{L})$
for the height $(\mathcal{L},\phi)^{n+1}$ and $\hat{h}_{\phi}(\mathcal{X},\mathcal{L})$
for the normalized height: 
\[
h_{\phi}(\mathcal{X},\mathcal{L}):=(\mathcal{L},\phi)^{n+1},\,\,\,\hat{h}_{\phi}(\mathcal{X},\mathcal{L}):=\frac{(\mathcal{L},\phi)^{n+1}}{(n+1)L^{n}}.
\]
The definition of $\hat{h}_{\phi}(\mathcal{X},\mathcal{L})$ is made
so that 
\[
\hat{h}_{\phi+\lambda}(\mathcal{X},\mathcal{L})=\hat{h}_{\phi}(\mathcal{X},\mathcal{L})+\lambda/2,\,\,\,\text{if }\lambda\in\R
\]
We also recall that, given two continuous psh metrics $\phi$ and
$\phi_{0}$ on the complexification $L\rightarrow X$ of $\mathcal{L}\rightarrow\mathcal{X},$
we have that 
\begin{equation}
2h\left(\mathcal{L},\phi\right)-2h\left(\mathcal{L},\phi_{0}\right)=\mathcal{E}(\phi,\phi_{0}):=\frac{1}{(n+1)!}\int_{X}(\phi-\phi_{0})\sum_{j=0}^{n}(dd^{c}\phi)^{j}\wedge(dd^{c}\phi_{0})^{n-j}.\label{eq:def of primitive}
\end{equation}

\subsection{\label{subsec:The-canonical-height}The canonical height of an arithmetic
log Fano variety}

We define the \emph{canonical height} $h_{\text{can}}(\mathcal{X},\mathcal{D})$
of an arithmetic log Fano variety $(\mathcal{X},\mathcal{D})$ by
\[
h_{\text{can}}(\mathcal{X},\mathcal{D}):=\sup\left\{ h_{\phi}\left(-\mathcal{K}_{(\mathcal{X},\mathcal{D})}\right):\,\text{\ensuremath{\phi\,\,\text{\ensuremath{\text{cont. }}psh,}}}\int_{X}\mu_{\phi}=1\right\} 
\]
(when $\mathcal{D}=0$ we shall use the short hand $h_{\text{can}}(\mathcal{X})$
for $h_{\text{can}}(\mathcal{X},0)$). As shown precisely as in the
case $\mathcal{D}=0,$ considered in \cite{a-b}, $h_{\text{can}}(\mathcal{X},\mathcal{D})<\infty$
iff the corresponding log Fano variety $(X,\Delta)$ over $\C$ is
K-semistable. Moreover, $(X,\Delta)$ is K-polystable iff the sup
defining $h_{\text{can}}(\mathcal{X},\mathcal{D})$ is attained at
some continuous metric $\phi,$ namely a log Kähler-Einstein metric.
Hence, if $(X,\Delta)$ is K-polystable, then the canonical height
$h_{\text{can}}(\mathcal{X},\mathcal{D})$ is computed by any volume-normalized
log Kähler-Einstein metric. 

\section{\label{sec:Toric-log-Fano}Toric log Fano varieties}

A log pair $(X,D)$ over $\C$ is said to be\emph{ toric }if $X$
and $D$ are toric, i.e. if $X$ is toric and the $\Q-$divisor $D$
is invariant under the torus action on $X.$ Any toric log Fano variety
admits a canonical integral model $(\mathcal{X},\mathcal{D})$ which
is log Fano (see \cite[Section 2]{ma} and \cite[Def 3.5.6]{b-g-p-s}).
In this section we will prove the following
\begin{thm}
\label{thm:main log toric}Let $(\mathcal{X},\mathcal{D})$ be the
canonical integral model of a K-semistable toric log Fano variety
$(X,D).$ Conjecture \ref{conj:height log intro} holds for $(\mathcal{X},\mathcal{D})$
under anyone of the following conditions:
\begin{itemize}
\item $n\leq3$ and $X$ is $\Q-$factorial (equivalently, $X$ has at worst
abelian quotient singularities) 
\item $X$ is not Gorenstein or has some abelian quotient singularity
\end{itemize}
\end{thm}

We start by introducing some notation, following \cite{ber-ber}.
Given a toric log Fano variety $(X,D)$ set $L=-(K_{X}+\Delta)$ and
denote by $P$ the corresponding moment polytope in $\R^{n}.$ Then
\begin{equation}
P=\left\{ p\in\R^{n}:\,\,\left\langle l_{F},p\right\rangle \geq-a_{F},\,\,\forall F\right\} ,\label{eq:P in log Fano case}
\end{equation}
where $a_{F}\in]0,1]$ (generalizing the Fano case when $a_{F}=1\forall F$;
see \cite{ber-ber}) and $l_{F}$ is a primitive integer vector. As
shown in \cite{ber-ber} $(X,\Delta)$ is K-semistable iff $0$ is
the barycenter of $P$ iff the log Ding functional $\mathcal{D}_{\psi_{P}}$
is bounded from below. Moreover, the infimum of $\mathcal{D}_{\psi_{P}}$
is attained at a $T-$invariant psh metric $\phi$ on $L.$ We will
identify the metric $\phi$ with a continuous convex function on $\R^{n}$
as in \cite{a-b}. More precisely, on $(\mathbb{C}^{*})^{n}\hookrightarrow X$,
let $x_{i}=\log(|z_{i}|^{2})$. Trivializing $-(K_{X}+\Delta)$ with
$\frac{\mathrm{d}z_{1}}{z_{1}}\wedge...\wedge\frac{\mathrm{d}z_{n}}{z_{n}}\otimes s_{U}e_{\Delta}$
over $U=(\mathbb{C}^{*})^{n}$, and abusing notation slightly, we
let $\phi(x):=\phi_{U}(z)$ in the chosen trivialization over $U=(\mathbb{C}^{*})^{n}$.
Then $\phi$ as a function of $x$ is a continuous convex function
on $\mathbb{R}^{n}$ and as in formula 3.8 in \cite{a-b}, we still
have that 
\[
\mathcal{D}_{\psi_{P}}(\phi)=\int_{P}\phi^{*}dy/V-\log\int_{\R^{n}}e^{-\phi(x)}dx-n\log\pi,\,\,\,V:=\text{vol\ensuremath{(P)}}
\]
 (since the support of $\mathcal{D}$ is contained in the complement
of $(\C^{*})^{n}$ in $X).$ Thus the inequality in Proposition 3.7
in \cite{a-b} generalizes to the canonical toric model $\mathcal{L}$
of $L$ (which coincides with $-\mathcal{K}_{(\mathcal{X},\mathcal{D})}):$
\begin{equation}
2\widehat{\mathrm{vol}}_{\chi}\left(-\mathcal{K}_{(\mathcal{X},\mathcal{D})},\phi\right)\leq-\mathrm{vol}(X,\Delta)\log\left(\frac{\mathrm{vol}(X,\Delta)}{(2\pi^{2})^{n}}\right)\,\,\mathrm{vol}(X,\Delta):=\mathrm{vol}(-K_{(X,\Delta}).\label{eq:universal log bound}
\end{equation}

We will first prove Theorem \ref{thm:main log toric} in the case
that $X=\P^{n},$ using the following lemma, formulated in terms of
the divisor $D_{0}$ cut out by the $T_{c}-$invariant element of
$H^{0}(X,-K_{X})$ (given by $\frac{\mathrm{d}z_{1}}{z_{1}}\wedge...\wedge\frac{\mathrm{d}z_{n}}{z_{n}}$
over $(\mathbb{C}^{*})^{n}$). In other words, 
\[
D_{0}=\sum_{F}D_{F},
\]
 where $D_{F}$ is the irreducible divisor corresponding to the facet
$F$ of the moment polytope corresponding to $X$ (see \cite{ber-ber}).
The lemma is a special case of formula Proposition 3.12 from \cite{a-b}. 
\begin{lem}
\label{lem:scaled polytope}Let $\mathcal{X}$ be the canonical integral
model of an $n-$dimensional K-semistable toric Fano variety $X$
and denote by $D_{0}$ the standard anti-canonical divisor on $X.$
Then 
\[
\frac{(\overline{-\mathcal{K}_{(\mathcal{X},(1-t)\mathcal{D}_{0})}})^{n+1}/(n+1)!}{\left(-(K_{X}+(1-t)D_{0})\right)^{n}/n!}=\frac{(\overline{-\mathcal{K}_{\mathcal{X}}})^{n+1}/(n+1)!}{\left(-K_{X}\right)^{n}/n!}-\frac{1}{2}\log(t^{n})\,\,\,\,\,\,\,t^{n}=\left(\frac{\left(-(K_{X}+(1-t)D_{0})\right)^{n}}{\left(-K_{X}\right)^{n}}\right)
\]
 with respect to the volume-normalized Kähler-Einstein metrics. 
\end{lem}

We next deduce the following
\begin{lem}
\label{lem: toric case when X is Pn}Let $(\mathcal{X},\mathcal{D})$
be a toric K-semistable log Fano variety such that $\mathcal{X}=\P_{\Z}^{n}.$
Then $(\overline{-\mathcal{K}_{(\mathcal{X},\mathcal{D})}})^{n+1}\leq(\overline{-\mathcal{K}_{\P_{\Z}^{n}}})^{n+1}$
with equality iff $\mathcal{D}=0.$
\end{lem}

\begin{proof}
First observe that there exists $t\in[0,1]$ such that $\mathcal{D}=(1-t)\mathcal{D}_{0}=:\mathcal{D}_{t}.$
This is a special case of \cite[Cor 1.6]{fu2}, which applies to $\P^{n},$
in any dimension $n,$ using that toric log Fano varieties are never
uniformly K-stable. It will thus be enough to show that $t\mapsto(\overline{-\mathcal{K}_{(\P_{\Z}^{n},\mathcal{D}_{t})}})^{n+1}$
is increasing on $[0,1]$ (and thus its maximum is attained at $t=1).$
By the previous lemma
\[
\frac{2(\overline{-\mathcal{K}_{(\P_{\Z}^{n},\mathcal{D}_{t})}})^{n+1}/(n+1)!}{(-K_{\P^{n}})^{n}/n!}=t^{n}2\frac{(\overline{-\mathcal{K}_{\P_{\Z}^{n}}})^{n+1}/(n+1)!}{(-K_{\P^{n}})^{n}/n!}-t^{n}\log(t^{n}).
\]
 Differentiating wrt $(t^{n})$ reveals that the right hand side above
is increasing with respect to $t$ iff $2\frac{(\overline{-\mathcal{K}_{\P_{\Z}^{n}}})^{n+1}/(n+1)!}{(-K_{\P^{n}})^{n}/n!}\geq1.$
The latter inequality is indeed satisfied, as follows from the explicit
formula \ref{eq:expl formul on p n}. 
\end{proof}
Combining the universal bound \ref{eq:universal log bound} with Lemma
3.8 from \cite{a-b}, all that remains to prove Theorem \ref{thm:main log toric}
is to establish the ``logarithmic gap hypothesis'' 
\begin{equation}
\text{vol}(X,\Delta)\leq\text{vol}(\P^{n-1}\times\P^{1}).\label{eq:log gap}
\end{equation}

\begin{prop}
The logarithmic gap hypothesis holds for all toric K-semistable log
Fano varieties (manifolds) $(X,\Delta)$ such that $X\neq\P^{n}$
iff the following bound holds for all Fano varieties (manifolds) $X\neq\P^{n}$
\begin{equation}
\text{\ensuremath{S(X)}}\leq\mathrm{vol}(\P^{n-1}\times\P^{1})\,\,\,\,S(X):=\sup\left\{ \mathrm{vol}(-(K_{X}+\Delta)):\,(X,\Delta)\,\text{K-semistable}\right\} .\label{eq:sup log}
\end{equation}
The ``logarithmic gap hypothesis'' holds for all log Fano varieties
$(X,\Delta)$ such that $X$ is $\Q-$factorial and of dimension $n\leq3$
and for any dimensions $n$ if $X$ has some abelian quotient singularity
or if $X$ is not Gorenstein.
\end{prop}

\begin{proof}
Since, trivially, $\text{vol}(X,\Delta)\leq S(X)$ the first equivalence
follows directly from the definitions. Next, let us show the last
statement of the proposition, first assuming that $X$ is singular,
which means that the moment polytope $P$ of $(X,\Delta)$ is ``singular''
in the sense that there exists a vertex of $\partial P$ such that
the corresponding primitive vectors $l_{F_{1}},...,l_{F_{n}}$ do
not generate $\Z^{n}.$ It follows from the proof of Lemma 3.9 from
\cite{a-b} that
\[
\mathrm{vol}(P)\leq\frac{1}{2}(n+1)^{n}/n!\leq\text{vol}(\P^{n-1}\times\P^{1}))
\]
 Indeed, since $a_{F}\leq1$ the first inequality follows from the
inequality (3.13) from \cite{a-b}, using that $\delta\geq2,$ according
to the singularity assumption on $P$ (for the the second inequality
see formula (3.14) from \cite{a-b}). All that remains is thus to
show the bound \ref{eq:sup log} for $S(X)$ when $n\leq3$ and $X$
is non-singular. First assume that $n=2.$ This means, by classical
classification results, that $X$ is either $\P^{1}\times\P^{1}$
or the blow-up $X^{(m)}$ of $\P^{2}$ in $m$ points for $m\leq3.$
But $(-K_{X^{(m)}})^{2}=(-K_{\P^{2}})^{2}-m$ and thus $\text{vol}(X_{1})\leq4=\text{vol}(\P^{n-1}\times\P^{1}),$
proving the bound \ref{eq:sup log}. Finally, consider the case when
$n=3.$ Starting with the trivial bound $\text{vol}(X,\Delta)\leq\text{vol}(X)$
it follows the classification \cite{ob} of all non-singular toric
Fano varieties of dimension 3 that it is enough to show that the bound
\ref{eq:sup log} holds when $X$ is $\P^{3}$ blown-up in one point
or $\P(\mathcal{O}(1)\oplus\mathcal{O}(2))$ (whose degrees are $56$
and $62,$ respectively). According to the following proposition the
corresponding invariants $S(X)n!$ are, approximately, given by $41.8$
and $30.3$, respectively, which are well below the degree $54$ of
$\P^{2}\times\P^{1},$ as desired.
\end{proof}

\subsection{\label{subsec:The-logarithmic-gap}The invariant $\text{\ensuremath{S(X)}}$
for $n\protect\leq3$}

In the proof above we used the following result.
\begin{prop}
After rounding to the nearest decimal place the invariant $n!S(X)$
(formula \ref{eq:sup log}) is given by $41.8$ and $30.3$ when $X$
equals $\mathbb{P}^{3}$ blown up in one point and $\mathbb{P}(\mathcal{O}\oplus\mathcal{O}(2)),$
respectively.
\end{prop}

\begin{proof}
Given a convex subset $P$ of $\R^{n}$ let
\[
s(P):=\sup\left\{ \text{vol}(P_{0}):\,P_{0}\subset P,\,b_{P_{0}}=0\right\} ,
\]
 where $P_{0}$ is a closed subset of $P$ with barycenter $b_{P_{0}}$
at the origin. We will compute $s(P)$ when $P$ is the moment polytope
of the manifolds $X$ appearing in the proposition, showing at the
same time that $s(P)=S(X).$ The moment polytopes $P$ of both $\mathbb{P}^{3}$
blown up in one point and $\mathbb{P}(\mathcal{O}\oplus\mathcal{O}(2))$
are of the form a simplex, with a simplex subset removed, by chopping
off a vertex (see ID 20 and ID7 in the database \cite{ob})). After
a general linear transformation, they are of the form $(a\Delta_{3}-\mathbf{1})-(b\Delta_{3}-\mathbf{1})$
where $\Delta_{3}$ is the standard unit simplex in dimension three,
$\mathbf{1}$ is the vector with all ones and $a$ and $b$ are positive
real numbers. For $\mathbb{P}^{3}$ blown up in one point we can transform
the moment polytope to $(4\Delta_{3}-\mathbf{1})\backslash(2\Delta_{3}-\mathbf{1})$
and for $\mathbb{P}(\mathcal{O}\oplus\mathcal{O}(2))$ we get $(5\Delta_{3}-\mathbf{1})\backslash(\Delta_{3}-\mathbf{1})$.
In the first case, the linear transformation is unimodular, but in
the second case the transformation has determinant $2.$ This will
not matter when computing $s(P)$ as long as we correct for the non-unit
determinant. Next we compute the barycenter $b_{P}$ of these polytopes,
a simple task using the explicit barycenter of the standard unit simplex,
$b_{\Delta_{n}}=\mathbf{1}/(n+1),$ and then scaling and linearity
properties of the volume times the barycenter. The barycenter of $(a\Delta_{3}-\mathbf{1})\backslash(b\Delta_{3}-\mathbf{1})$
is given by $\frac{a^{3}/3!(a/4-1)-b^{3}/3!(b/4-1)}{a^{3}/3!-b^{3}/3!}\mathbf{1}$.
Next we use a general fact, to be proved in the lemma below, stating
that the closed subset $P'$ of $P$ which maximizes volume, with
the relaxed constraint 
\begin{equation}
b_{P'}\cdot\mathbf{1}=0\label{eq:constraint}
\end{equation}
is the one given by $P\cap H$ where $H$ is a half-space with normal
$\mathbf{1}$. In our case, by symmetry, this $P'$ automatically
satisfies the stronger constraint $b_{P'}=0.$ Moreover, since the
boundary of $P\cap H$ is parallel to a facet of $P$ it corresponds
to a divisor $\Delta$ on $X$ defining a log Fano pair $(X,\Delta).$
Thus $(X,\Delta)$ is also the K-semistable log Fano pair realizing
the sup in the definition of $S(X),$ showing that $s(P)=S(X).$ We
can find $H$ by imposing the constraint. We introduce the weight
$w$ such that 
\[
P\cap H=((a-w)\Delta_{3}-\mathbf{1})\backslash(b\Delta_{3}-\mathbf{1}).
\]
From here it is clear that if $b_{P'}\cdot\mathbf{1}=0,$ then, in
fact, the entire barycenter will vanish and the condition $b_{P'}\cdot\mathbf{1}=0$
turns into the following fourth order polynomial equation for $w:$
\[
(a-w)^{3}/3!((a-w)/4-1)-b^{3}/3!(b/4-1)=0.
\]
The solution $w$ and the corresponding value $s(P)$ for $\mathbb{P}^{3}$
blown up in one point, is given by $w=\frac{2}{3}(5-\frac{4}{\sqrt[3]{19-3\sqrt{33}}}-\sqrt[3]{19-3\sqrt{33}})$
and $\mathrm{n!s(P)=n!vol}(P')=((4-w)^{3}-2^{3})\approx41.8$ and
for $\mathbb{P}(\mathcal{O}\oplus\mathcal{O}(2))$, $w=(4-\sqrt[3]{\frac{4}{2-\sqrt{2}}}-\sqrt[3]{2(2-\sqrt{2})})$
and $\mathrm{n!S(P)=\frac{1}{2}n!vol}(P')=\frac{1}{2}((5-w)^{3}-1^{3})\approx30.3$,
where we have corrected for the non-unimodular transformation used
in the second case. 
\end{proof}
In the above proof we used the following
\begin{lem}
Let $P$ be a closed subset of $\mathbb{R}^{n}$ with the origin as
an interior point. Given $v\in\mathbb{R}^{n}$ assume that $\int_{P}x\cdot v>0$.
Then the maximum
\[
\max_{Q\subset P:\int_{Q}v\cdot x\mathrm{d\lambda(x)=0}}\int_{Q}\mathrm{d}\lambda
\]
is attained at $Q=P\cap H$ with $H$ a closed half-space with outward
pointing normal $v$. Here $\mathrm{d}\lambda$ is Lebesgue measure. 
\end{lem}

\begin{proof}
Without loss of generality we can assume that $v=(0,...,0,1).$ Denote
by $(x_{1},x_{2},...,x_{n-1},y)$ the coordinates on $\mathbb{R}^{n}.$
Since the origin is an interior point of $P$ and $\int_{P}x\cdot v>0$
there is a closed half-space $H$ as in the lemma satisfying $\int_{P\cap H}y\mathrm{d\lambda=0}$.
Hence, any candidate $Q$ for the maximum in question satisfies $\int_{P\cap H}y\mathrm{d\lambda=\int_{Q}y\mathrm{d\lambda}.}$
Subtracting the left hand side from the right hand side and vice versa
yields $\int_{P\cap H\backslash Q}y\mathrm{d\lambda=\int_{Q\backslash P\cap H}y\mathrm{d\lambda}.}$
Since $\sup_{P\cap H\backslash Q}y\leq\inf_{Q\backslash(P\cap H)}y$
it follows that $\text{vol\ensuremath{(P\cap H\backslash Q)\geq\text{vol(\ensuremath{Q\backslash P\cap H)}}}}$
which, in turn, implies that $\text{vol\ensuremath{(P\cap H)\geq\text{vol(\ensuremath{Q)}}},}$
as desired. 
\end{proof}
In fact, with just a slight variation of the argument above, any maximizer
must be of the special form above and, in addition, assuming connectedness
of $P$, the maximizer is unique. The proof of the previous proposition
thus reveals that the unique toric divisor $\Delta$ on $X$ realizing
the sup defining the invariant $S(X)$ is a multiple of the prime
divisor $D_{F}$ defined by the zero-section of $\mathbb{P}(\mathcal{O}\oplus\mathcal{O}(2))\rightarrow\P^{2}$
and hyperplane ``at infinity'' in $\P^{3}$ blown up at the origin
in $\C^{3}\subset\P^{3},$ respectively (i.e the zero-section of $\mathbb{P}(\mathcal{O}\oplus\mathcal{O}(1))\rightarrow\P^{2}).$
A similar argument also applies when $X$ is the blow-up of $\P^{2}$
at the origin in $\C^{2}$ (i.e. the first Hirzebruch surface $\mathbb{P}(\mathcal{O}\oplus\mathcal{O}(1))\rightarrow\P^{2}).$
The unique maximizer for the invariant $S(X)$ is then a log Fano
pair $(X,\Delta)$ for a multiple of the hyperplane $D$ ``at infinity''
(i.e. the zero-section of $\mathbb{P}(\mathcal{O}\oplus\mathcal{O}(1))\rightarrow\P^{2}).$
Interestingly, this K-polystable log pair $(X,\Delta)$ was also singled
out in \cite[Cor 1.5]{r-z} by the following rigidity property (answering
a question of Cheltsov): it admits a rigid Kähler-Einstein metric
in the sense that for any other multiple $cD$ the log pair $(X,cD)$
does not admit a Kähler-Einstein metric. The same rigidity property
holds for the two three-dimensional log pairs discussed above (since
there is a unique half-space $H$ satisfying the constraint in formula
\ref{eq:constraint}).

\subsection{Estimates on the canonical height}

Theorem 1.3 from \cite{a-b} (and its corollary) generalizes directly
to the case of log Fano pairs and their Kähler-Einstein metrics in
any relative dimension $n$ (with the same proof, by letting $P$
be the moment polytope corresponding to $(X,\Delta)):$ 
\[
\frac{1}{2}\mathrm{vol}(X,\Delta)\log\left(\frac{n!m_{n}\pi^{n}}{\mathrm{vol}(X,\Delta)}\right)\leq\frac{h_{\text{can}}(\mathcal{X},\mathcal{D})}{(n+1)!}\leq\frac{1}{2}\mathrm{vol}(X,\Delta)\log\left(\frac{(2\pi)^{n}\pi^{n}}{\mathrm{vol}(X,\Delta)}\right)
\]
Interestingly, Lemma \ref{lem:scaled polytope} reveals that the family
of log Fano pairs $(\mathcal{X},\mathcal{D})$ appearing in the lemma
may be explicitly expressed in terms of the algebro-geometric volume
$\text{vol}(X,\Delta)$ in the same functional form as the one appearing
in the previous upper and lower bounds:
\[
\frac{(\overline{-\mathcal{K}_{(\mathcal{X},\mathcal{D})}})^{n+1}}{(n+1)!}=\frac{1}{2}\text{vol}(X,\Delta)\log\left(\frac{be^{2a}}{\text{vol}(X,\Delta)}\right)
\]
 with $a:=\frac{(\overline{-\mathcal{K}_{\mathcal{X}}})^{n+1}/(n+1)!}{\left(-K_{X}\right)^{n}/n!}$
and $b=\text{vol}(X).$ 

\section{\label{sec:Hyperplane-arrangements}Hyperplane arrangements}

In this Section we prove Theorem \ref{thm:hyperplane intro} concerning
hyperplane arrangements. Recall that a log Fano pair $(X,\Delta)$
is called a \emph{log Fano hyperplane arrangement }if $X=\mathbb{P}^{n}$
and $\Delta=\sum_{i=1}^{m}w_{i}H_{i}$ where $w_{i}\in\Q_{>0}$ and
the $H_{i}$ are distinct hyperplanes. Furthermore we will call $(X,\Delta)$
\emph{simple normal crossing}, abbreviated snc, if the support of
$\Delta$ has simple normal crossings. 

For an snc log Fano hyperplane arrangement, if $m=n+1$ and all the
weights $w_{i}$ are equal, then $(X,\Delta)$ is a toric log-pair
(see Lemma \ref{lem:scaled polytope}). The following lemma shows
that for given hyperplanes $H_{1},...,H_{m}$ and a fixed volume $\text{vol}(X,\Delta)$,
the ``toric'' weights form the vertices of the convex polytope of
all weights $w_{i}$ corresponding to K-semistable $(X,\Delta).$ 
\begin{lem}
Fix $m\geq1$ and a real number $0<D\leq(-K_{\mathbb{P}^{n}})^{n}=(n+1)^{n}.$
Let as before for a real $m$-tuple $w$, $\Delta=\sum_{i=1}^{m}w_{i}H_{i}$
for distinct hyperplanes $H_{i}$. Then the set of weights
\[
S=\{w\in\R^{n}:(-(K_{\mathbb{P}^{n}}+\Delta))^{n}=D\mathrm{\ and\ }(\mathbb{P}^{n},\Delta)\mathrm{\ is\ K-semistable}\}
\]
is either empty or $m\geq n+1$ and $S$ is a polytope with $\binom{n}{m}$
vertices given by any reordering of the tuple $w_{1}=w_{2}=...=w_{n+1}=\frac{1}{m}(n+1-D^{1/n})$
, $w_{l}=0$ $\forall l>n+1$.\label{lem: stability polytope}
\end{lem}

\begin{proof}
By \cite{fu2}, for $w\in\mathbb{R}^{n}$ and $\Delta=\sum_{i=1}^{m}w_{i}H_{i}$
, $(\mathbb{P}^{n},\Delta)$ is K-semistable and log Fano if and only
if $w$ is in the convex set $C$ defined by the following inequalities:

\begin{align}
0 & \leq w_{i}<1\ \forall i=1,...,m\nonumber \\
k & \sum_{i=1}^{m}w_{i}\geq(n+1)\sum_{j=1}^{k}w_{i_{j}}\ \forall1\leq k\leq n\ \forall1\leq i_{1}<...<i_{k}\leq m.\label{eq:def of C}
\end{align}
Here it should be noted that in fact, it suffices to consider the
second inequality for index combinations $i_{j}$ of length 1. The
other inequalities for larger index combinations follows. Hence, K-semistability
of $(\mathbb{P}^{n},\Delta)$ is equivalent to 

\begin{align}
0 & \leq w_{i}<1\ \forall i=1,...,m\nonumber \\
w_{i} & \leq\frac{1}{n+1}\sum_{j=1}^{m}w_{j}\ \forall i=1,...,m.\label{eq:weight cond}
\end{align}
Fix $m$ and $D$ as in the statement of the theorem. The goal is
to understand the intersection of the above set with the set $\{w:-(K_{\P^{n}}+\Delta)=D\}.$
Note first that 
\[
(-(K_{\P^{n}}+\sum_{i=1}^{m}w_{i}H_{i}))^{n}=(n+1-\sum_{i=1}^{m}w_{i})^{n}.
\]
Let $C:=n+1-D^{1/n}$, so that $\{w:-(K_{\P^{n}}+\Delta)=D\}=\{w:\sum_{i=1}^{m}w_{i}=C\}$.
Thus with S defined as in \ref{lem: stability polytope}
\begin{align*}
S=\{w: & \sum_{i=1}^{m}w_{i}=C\\
 & 0\leq w_{i}<1\ \forall i=1,...,m\\
 & w_{i}\leq\frac{C}{n+1}\ \forall i=1,...,m\}
\end{align*}
Observe that since $0\leq C<n+1$, the inequality $w_{i}<1$ is superfluous.
After a convenient rescaling we get

\begin{align*}
\frac{n+1}{C}S=\{w: & \sum_{i=1}^{m}w_{i}=n+1\\
 & 0\leq w_{i}\leq1\ \forall i=1,...,m\}.
\end{align*}
Clearly if $m<n+1$, $\frac{n+1}{C}S$ is empty. For $m\geq n+1$,
any vertex of $\frac{n+1}{C}S$ is given by the intersection of $\text{\ensuremath{\frac{n+1}{C}S}}$
with some collection of the inequalities put to equality. But clearly
all such points must be of the form of having $n+1$ ones and $m-(n+1)$
zeros. And on the other hand any such point is a vertex. 
\end{proof}
Fixing the volume $\text{vol}(X,\Delta)$ is, when $X=\P^{n},$ tantamount
to fixing the isomorphism class of the $\Q-$line bundle $-(K_{X}+\Delta)$
(since the rank of the Picard group of $\P^{n}$ is one). The following
lemma shows that, in this case, the maximal height is convex with
respect to the weights of $\Delta.$ 
\begin{lem}
\label{lemma:convexity of height}Consider an arithmetic Fano variety
$\mathcal{X}$ and a curve $t\mapsto(\mathcal{X},\mathcal{D}_{t})$
of arithmetic log Fano varieties where $\mathcal{D}_{t}=\sum_{i=1}^{m}w_{i}(t)\mathcal{D}_{i}$
for some $m\geq1$, irreducible divisors $\mathcal{D}_{i}$ over $\Z$
and $w:[0,1]\rightarrow\R^{m}$ an affine function. Additionally assume
that all the $\mathcal{D}_{t}$ are linearly equivalent, which equivalently
means that $-(\mathcal{K}+\mathcal{D}_{t})$ isomorphic to $\mathcal{L}$
for a line bundle $\mathcal{L}\rightarrow\mathcal{X}$ independent
of $t.$ Then the function $h:[0,1]\rightarrow]-\infty,\infty]$ defined
as
\begin{equation}
t\mapsto h_{\text{can}}(\mathcal{X},\mathcal{D}_{t})\label{eq: convex height equation}
\end{equation}
is strictly convex. Equivalently the function $t\mapsto\hat{h}_{\text{can}}(\mathcal{X},\mathcal{D}_{t})$
is strictly convex.
\end{lem}

\begin{proof}
By assumption we can identify $-(\mathcal{K}+\mathcal{D}_{t})$ with
$\mathcal{L}$ for a line bundle $\mathcal{L}$ independent of $\phi.$
Thus the height $h_{\text{\ensuremath{\phi}}}(\mathcal{X},\mathcal{D}_{t})$
for a fixed metric on $\mathcal{L}$ is independent of $t.$ Likewise,
$h_{\text{can}}(\mathcal{X},\mathcal{D}_{t})$ coincides with $\hat{h}_{\text{can}}(\mathcal{X},\mathcal{D}_{t})$
up to multiplication by a constant independent of $t.$ Next, express
\[
\hat{h}_{\text{can}}(\mathcal{X},\mathcal{D}_{t})=\sup_{\phi}\hat{h}(\mathcal{X},\mathcal{D}_{t})+\frac{1}{2}\log\int_{X}\mu_{(\phi,\mathcal{D}_{t})},
\]
 where the sup ranges over all continuous psh metrics on $\mathcal{L}.$
Introducing an arbitrary volume form $\mathrm{d}V$ on $X$ we can
rewrite
\[
\int_{X}\mu_{(\phi,\mathcal{D}_{t})}=\int_{X}\exp(-\phi-\sum_{i=1}^{m}w_{i}(t)\psi_{D_{i}}-\log\mathrm{d}V)dV.
\]
By Hölder's inequality this expression is convex in $t,$ since $w_{i}(t)$
is affine. It is even strictly convex since the $D_{i}$ are distinct.
This means that $\hat{h}_{\text{can}}(\mathcal{X},\mathcal{D}_{t})$
is the supremum over a set independent of $t$, of a collection of
strictly convex functions and thus is itself, strictly convex. 
\end{proof}
The above lemmas will reduce the proof of Theorem \ref{thm:hyperplane intro}
to the case of when the support of $\Delta$ consists of $n+1$ distinct
hyperplanes. The following lemma shows that we can further reduce
to the case when these hyperplanes are the standard toric ones. 
\begin{lem}
\label{lem: reduction to the standard toric case}Assume that the
log pair $(\mathbb{P}_{\Z}^{n},\mathcal{D})$ is isomorphic to the
standard toric one $(\mathbb{P}_{\Z}^{n},\mathcal{D}_{0})$ over $\C.$
Then, for any $t\in[0,1],$ 
\[
h_{\text{can}}(\P_{\Z}^{n},(1-t)\mathcal{D})\leq h_{\text{can}}(\P_{\Z}^{n},(1-t)\mathcal{D}_{0})
\]
\end{lem}

\begin{proof}
Denote by $s_{0},...,s_{n}$ the integral sections of $\mathcal{O}(1)$
cutting out the irreducible components of $\mathcal{D}.$ We can express
$s_{i}:=\sum_{j}A_{ij}x_{j}$ with $A_{ij}\in\Z,$ where $(x_{0},...,x_{n})$
are the standard affine coordinates on $\C^{n+1}.$ It will be enough
to show that 
\begin{equation}
\hat{h}_{\text{can}}(\P_{\Z}^{n},(1-t)\mathcal{D})=\hat{h}_{\text{can}}(\P_{\Z}^{n},(1-t)\mathcal{D}_{0})+(t-1)\log\left|\det A\right|\label{eq:desired id}
\end{equation}
 (since $t\leq1$ and $\left|\det A\right|\geq1).$ To this end, denote
by $F$ the invertible $\C-$linear map from $\C_{x}^{n+1}\rightarrow\C_{y}^{n+1}$
satisfying 
\[
F^{*}(\sum_{j}A_{ij}y_{j})=x_{i}
\]
 (the existence of $F$ is equivalent to the invertibility of the
matrix $A$ which, in turn, is equivalent to the assumption about
an isomorphism over $\C).$ We will use the same symbol $F$ for the
induced map $\P^{n}\rightarrow\P^{n}$ and its standard lift to $\mathcal{O}(1)$
(as well as its tensor powers). By basic linear algebra 
\[
F^{*}dy_{0}\wedge...\wedge dy_{n}=(\det A)^{-1}dx_{0}\wedge...\wedge dx_{n}.
\]
Now, observe that $-K_{(\P_{\Z}^{n},(1-t)\mathcal{D})}\simeq t\mathcal{O}(n+1),$
using standard isomorphisms over $\Z.$ In particular, any given metric
$\phi$ on $\mathcal{O}(n+1)$ induces a metric $t\phi$ on $-K_{(\P_{\Z}^{n},(1-t)\mathcal{D})}$
and thus, using the log pair $(\P_{\Z}^{n},(1-t)\mathcal{D}),$ a
measure $F^{*}\mu_{t\phi}$ on $\P_{\Z}^{n}.$ Note that

\begin{equation}
\int_{\P^{n}}\mu_{t\phi}=\left|\det A\right|^{-2}\int_{\P^{n}}\mu_{tF^{*}\phi},\label{eq:integrals same}
\end{equation}
 where $\mu_{tF^{*}\phi}$ is the measure associated to the metric
$tF^{*}\phi$ on $-K_{(\P_{\Z}^{n},(1-t)\mathcal{D}_{0})}$ (now using
the toric log pair $(\P_{\Z}^{n},(1-t)\mathcal{D}_{0})).$ Indeed,
consider the (singular) metric $\psi$ on $(n+1)\mathcal{O}(1)$ defined
by 
\[
\psi:=t\phi+(1-t)\sum_{i}\log|s_{i}|^{2}.
\]
 The measure $\mu_{t\phi}$ coincides with the measure $\mu_{\psi}$
attached to the metric on $-K_{\P^{n}}$ induced by $\psi.$ Hence,
denoting by $G$ the canonical lift of the map $F$ to $-K_{\P^{n}}$
we have, by definition, that 
\[
G^{*}\mu_{\psi}=\mu_{G^{*}\psi},
\]
\[
G^{*}\psi=F^{*}\psi+\log(\left|\det A\right|^{2}),\,\,\,\,F^{*}\psi:=tF^{*}\phi+(1-t)\sum_{i}\log|F^{*}s_{i}|^{2}.
\]
This proves formula \ref{eq:integrals same}. Next, note that, by
Lemma \ref{lem:append} in the appendix,
\[
\hat{h}(\mathcal{O}(n+1),G^{*}\phi)=\hat{h}(\mathcal{O}(n+1),\phi),
\]
Hence, we get, as above, that

\[
\hat{h}(\mathcal{O}(n+1),F^{*}\phi)=\hat{h}(\mathcal{O}(n+1),G^{*}\phi-\log(\left|\det A\right|))=
\]
\[
=\hat{h}(\mathcal{O}(n+1),G^{*}\phi)-\log(\left|\det A\right|^{2})=\hat{h}(\mathcal{O}(n+1),\phi)-\log(\left|\det A\right|)),
\]
giving 
\[
\hat{h}(\mathcal{O}(t(n+1)),t\phi)=\hat{h}(\mathcal{O}(t(n+1)),F^{*})(t\phi))+t\log(\left|\det A\right|).
\]
 All in all this means, that 
\[
\hat{h}(\mathcal{O}(t(n+1)),t\phi)+\log\int\mu_{t\phi}=\hat{h}(\mathcal{O}(t(n+1)),F^{*}t\phi)+\log\int\mu_{F^{*}t\phi}+(t-1)\log(\left|\det A\right|).
\]
 Taking the sup over all continuous metrics $\phi$ on $\mathcal{O}(n+1)$
with positive curvature thus concludes the proof of the desired identity
\ref{eq:desired id}.
\end{proof}

\subsection{Conclusion of the proof of Theorem \ref{thm:hyperplane intro}}

In the following we will use the notation $\mathcal{D}_{w}=\sum_{i=1}^{m}w_{i}H_{i}$
for $w\in\mathbb{R}^{m}$ and for fixed hyperplanes $H_{i}$ defined
over $\Z$. Let $(\P_{\Z}^{n},\mathcal{D}_{w'})$ be a K-semistable
snc log Fano hyperplane arrangement. Define for brevity $d:=(-(K_{\P^{n}}+\Delta(w')))^{n}$.
Set, as in Lemma \ref{lem: stability polytope}, $S=\{w\in\R^{n}:(-(K_{\mathbb{P}^{n}}+\Delta_{w}))^{n}=d\mathrm{\ and\ }(\mathbb{P}_{\Z}^{n},\mathcal{D}_{w})\mathrm{\ is\ K-semistable}\}$.
Consider the function $h(w)$ defined by
\[
h(w)=h_{\text{can}}(\P_{\Z}^{n},\mathcal{D}_{w}).
\]
 Restricted to the convex set $S$, $h|_{S}$ is convex by Lemma \ref{lemma:convexity of height}.
Next by Lemma \ref{lem: stability polytope}, $S$ is the convex hull
of weight vectors $(w^{k})_{k=1,...,\binom{m}{n+1}}$, each corresponding
to toric log Fano pairs,  equivalent to $(\mathbb{P}_{\Z}^{n},(1-t)\mathcal{D}_{0})$
over $\C,$ where $\mathcal{D}_{0}$ is the toric standard anti-canonical
divisor and $t$ is the unique number such that $(-(K_{\P^{n}}+(1-t)\mathcal{D}_{0})^{n}=d$.
By the convexity of $h$, 
\[
h(w')\leq\max_{k}h(w^{k})\leq h_{\text{can}}(\P_{\Z}^{n},(1-t)\mathcal{D}_{0})
\]

where the second inequality is the content of Lemma \ref{lem: reduction to the standard toric case}.
We have thus reduced to the standard toric case, which we have already
handled. Specifically, the bound \ref{eq:estimate by toric intro}
follows directly from Lemma \ref{lem:scaled polytope}. For Theorem
\ref{thm:hyperplane intro}, recall that it was observed in the proof
of Lemma \ref{lem: toric case when X is Pn} that the volume dependent
bound in \ref{eq:estimate by toric intro} is strictly increasing
with volume, so that a universal bound is uniquely given for maximal
volume, i.e. when $\Delta=0$, yielding the result. 

\section{Diagonal hypersurfaces}

In this section we will deduce Theorem \ref{thm:diagonal hypersurface intro}
from the results in the previous sections. The starting point of the
proof is the following analytic representation of the height:
\begin{lem}
\label{lem:(Restriction-formula)-Let} (Restriction formula) Let $\mathcal{X}$
be the subscheme of $\P_{\Z}^{n+1}$ cut out by a homogeneous polynomial
$s$ of degree $d$ with integer coefficients and $\phi$ a continuous
psh metric on $\mathcal{O}(d)\rightarrow\P_{\C}^{n+1}.$ Then the
height $h_{\phi}(\mathcal{X}_{d},\mathcal{O}(d))$ of the restriction
of $(\mathcal{O}(d),\phi)$ to $\mathcal{X}$ may be expressed as
\[
\frac{2h_{\phi}(\mathcal{X}_{d},\mathcal{O}(d))}{(n+1)!}=(n+2)\mathcal{E}(\phi,d\phi_{0})+\int_{\P^{n+1}}\log\left(\left\Vert s\right\Vert _{\phi}^{2}\right)\frac{(dd^{c}\phi)^{n+1}}{(n+1)!},
\]
 where $\phi_{0}$ is the Weil metric on $\mathcal{O}(1)$ and $\mathcal{E}$
is the functional defined by formula \ref{eq:def of primitive}, corresponding
to $\mathcal{O}(d)\rightarrow\P^{n+1}.$
\end{lem}

\begin{proof}
This is well-known, but for completeness we provide a proof. Consider
first the general situation where $\mathcal{X}$ is a subscheme (of
relative dimension $n)$ of a regular projective flat scheme $\mathcal{Y}$
cut out by a section $s$ of a relatively ample line bundle $\mathcal{L}\rightarrow\mathcal{Y}.$
Then, given a metric $\phi$ on the complexification $L$ of $\mathcal{L}\rightarrow\mathcal{Y},$
the restriction formula for arithmetic intersection numbers \cite[Prop 2.3.1]{b-g-s}
gives 
\begin{equation}
(\mathcal{L},\phi)^{n+2}\cdot\mathcal{Y}=(\mathcal{L},\phi)^{n+1}\cdot\mathcal{X}-\int_{Y}\log||s||_{\phi}(dd^{c}\phi)^{n+1}\label{eq:restriction formula}
\end{equation}
In particular, setting $\mathcal{Y}=\P_{\Z}^{n+1}$ and $\mathcal{L}=\mathcal{O}(d)$
gives and 
\[
\frac{2h_{\phi}(\mathcal{X}_{d},\mathcal{O}(d))}{(n+1)!}=(n+2)\frac{h_{\phi}(\P_{\Z}^{n+1},\mathcal{O}(d))}{(n+2)!}+\int\log||s||_{\phi}^{2}\frac{(dd^{c}\phi)^{n+1}}{(n+1)!}.
\]
 The proof is thus concluded by invoking the well-known fact that
$h_{\phi}(\P_{\Z}^{n+1},\mathcal{O}(d))/(n+2)!=\mathcal{E}_{\P^{n+1}}(\phi,\phi_{0}).$
For example, this is a special case of the toric formula in \cite[formula 3.7]{ab}.
\end{proof}
In general, if $\mathcal{X}$ is subscheme of $\P_{\Z}^{n+1}$ of
codimension one, then $\mathcal{K}_{\mathcal{X}}$ is well-defined
as line bundle over $\mathcal{X}.$ More precisely, by the adjunction
formula, there is an isomorphism of line bundles over $\Z,$
\[
\mathcal{K}_{\mathcal{X}}\simeq\left(\mathcal{K}_{\P_{\Z}^{n+1}}-\mathcal{O}(\mathcal{I}/\mathcal{I}^{2})\right)_{\mathcal{X}},
\]
 where $\mathcal{I}$ is the ideal sheaf cutting out $\mathcal{X}$
\cite[formula 1.6.2, page 8]{ko}. In particular, if $\mathcal{X}$
is cut out by a homogeneous polynomial $s$ of degree $d,$ then 
\begin{equation}
-\mathcal{K}_{\mathcal{X}}\simeq-\mathcal{K}_{\P_{\Z}^{n+1}}-\mathcal{O}(d)\simeq\mathcal{O}(n+2-d).\label{eq:adjunction isomo}
\end{equation}
Hence, $-\mathcal{K}_{\mathcal{X}}$ is relatively ample iff $d\leq n+1.$
Now assume that the complex variety $X$ defined by the complex points
of $\mathcal{X}$ is non-singular. Then, by the adjunction isomorphism
\ref{eq:adjunction isomo}, a metric $\phi_{0}$ on $\mathcal{O}(n+2-d)|_{X}$
may be identified with a metric on $-K_{X}.$

\subsection{Reduction to Fermat hypersurfaces}

Given integers $a_{i}$ consider the subscheme $\mathcal{X}_{a}$
of $\P_{\Z}^{n+1}$ cut out by the homogeneous polynomial

\[
s_{a}:=\sum_{i=0}^{n+1}a_{i}x_{i}^{d}.
\]
 Denote by $X_{a}$ the corresponding complex variety, which is non-singular
and consider the map
\begin{equation}
F_{a}(x):=(a_{0}^{-1/d}x_{0},...,a_{n+1}^{-1/d}x_{n+1}):\,\,\C^{n+2}\rightarrow\C^{n+2}.\label{eq:def of F a}
\end{equation}
We will identify the map $F_{a}$ with an automorphism of $\P^{n+1},$
admitting a standard lift to the total space of $\mathcal{O}(1)\rightarrow\P^{n+1}.$
We can then express
\[
X_{a}=F_{a}(X_{1}).
\]

\begin{prop}
\label{prop:red to fermat}Let $k$ be a positive integer and $\phi$
a metric on $\mathcal{O}(n+2-d)|_{X_{a}}.$ Then 
\[
2\hat{h}_{\phi}(\mathcal{X}_{a},\mathcal{O}(n+2-d))+\log\int_{X_{a}}e^{-\phi}=
\]
\[
=2\hat{h}_{F_{a}^{*}\phi}(\mathcal{X}_{1},\mathcal{O}(n+2-d))+\log\int_{X_{1}}e^{-F_{a}^{*}\phi}+(\frac{n+2-d}{(n+1)}-1)d^{-1}\sum_{i}\log(|a_{i}|^{2}).
\]
\end{prop}

The proof of the proposition follows from combining the following
two lemmas:
\begin{lem}
\label{lem:transf of height}Let $k$ be a positive integer and $\phi$
a metric on $\mathcal{O}(k).$ Then 
\[
2\widehat{h}_{\phi}(\mathcal{X}_{a},\mathcal{O}(k))=2\widehat{h}_{F_{a}^{*}\phi}(\mathcal{X}_{1},\mathcal{O}(k))+\frac{1}{d}\sum_{i}\log(|a_{i}|^{2})\frac{k}{(n+1)}
\]
\end{lem}

\begin{proof}
We will use that any continuous psh metric $\phi$ on $\mathcal{O}(k)|_{X}$
is the restriction of a continuous psh metric on $\mathcal{O}(k)\rightarrow\P^{n+1},$
that we shall denote by the same symbol $\phi$ \cite{c-g-z}. First
consider the case when $k=d$ and denote by $s_{a}$ the section of
$\mathcal{O}(d)$ cutting out the scheme $\mathcal{X}_{a}.$ By the
restriction formula (Lemma \ref{lem:(Restriction-formula)-Let})
\[
\frac{2}{(n+1)!}h_{\phi}(\mathcal{X}_{a},\mathcal{O}(k))=(n+2)\mathcal{E}_{\P^{n+1}}(\phi,d\phi_{0})+\int_{\P^{n+1}}\log|s_{a}|_{\phi}^{2}MA(\phi).
\]
 Rewriting
\[
\int_{\P^{n+1}}\log|s_{a}|_{\phi}^{2}MA(\phi)=\int_{\P^{n+1}}\log|(F_{a}^{-1})^{*}s_{1}|_{\phi}^{2}MA(\phi)=\int_{\P^{n+1}}\log|s_{1}|_{F_{a}^{*}\phi}^{2}MA(F_{a}^{*}\phi),
\]
thus reveals that 
\[
\frac{2}{(n+1)!}h_{F_{a}^{*}\phi}(X_{1},\mathcal{O}(k))-\frac{2}{(n+1)!}h_{\phi}(X_{a},\mathcal{O}(k))=(n+2)\left(\mathcal{E}_{\P^{n+1}}(F_{a}^{*}\phi,d\phi_{0})-\mathcal{E}_{\P^{n+1}}(\phi,d\phi_{0})\right)
\]
Now, denote by $G_{a}$ the standard lift of $F_{a}$ from $X$ to
$-K_{X}$ and its tensor powers. We can then express 
\[
G_{a}^{*}\phi=F_{a}^{*}\phi+c_{a},\,\,\,c_{a}:=\frac{k}{(n+2)}\frac{1}{d}\sum_{i}\log(|a_{i}|^{2})
\]
Indeed, 
\begin{equation}
G_{a}^{*}(e^{-\frac{n+2}{k}\phi}dzd\bar{z})=(e^{-\frac{n+2}{k}F_{a}^{*}\phi})F_{a}^{*}(dzd\bar{z})=(e^{-\frac{n+2}{k}F_{a}^{*}\phi})\prod_{i}|a_{i}|^{-2/d}(dzd\bar{z}).\label{eq:pull-back volume form}
\end{equation}
 Hence,
\[
2h_{F_{a}^{*}\phi}(\P^{n+1},\mathcal{O}(k))=2h_{G_{a}^{*}\phi-c_{a}}(\P^{n+1},\mathcal{O}(k))=2h_{G_{a}^{*}\phi}(\P^{n+1},\mathcal{O}(k))-c_{a}\frac{k^{n+1}}{(n+1)!}
\]
But, by Lemma \ref{lem:append} in the appendix
\[
\mathcal{E}_{\P^{n+1}}(G_{a}^{*}\phi,d\phi_{0})=\mathcal{E}_{\P^{n+1}}(\phi,d\phi_{0})
\]
Hence, $\frac{2}{(n+1)!}h_{F_{a}^{*}\phi}(X_{1},\mathcal{O}(k))-\frac{2}{(n+1)!}h_{\phi}(X_{a},\mathcal{O}(k))=$
\[
=(n+2)\left(\mathcal{E}_{\P^{n+1}}(F_{a}^{*}\phi,d\phi_{0})-\mathcal{E}_{\P^{n+1}}(\phi,d\phi_{0})\right)=-(n+2)c_{a}\frac{k^{(n+1)}}{(n+1)!}
\]
As a consequence,
\[
2\hat{h}_{F_{a}^{*}\phi}(X_{1},\mathcal{O}(k))-2\hat{h}_{\phi}(X_{a},\mathcal{O}(k))=-(n+2)c_{a}\frac{k^{n+1}}{(n+1)!}.\frac{1}{dk^{n}/n!}=-(n+2)c_{a}\frac{k}{d(n+1)}=
\]
\[
=-\frac{1}{d}\sum_{i}\log(|a_{i}|^{2})(n+2)\frac{k}{(n+2)}\frac{k}{d(n+1)}=-\frac{1}{d}\sum_{i}\log(|a_{i}|^{2})\frac{k^{2}}{d(n+1)}
\]
Since we have assumed that $k=d$ this means that 
\[
2\widehat{h}_{\phi}(X_{a},\mathcal{O}(k))-2\widehat{h}_{F_{a}^{*}\phi}(X_{1},\mathcal{O}(k))=\frac{1}{d}\sum_{i}\log(|a_{i}|^{2})\frac{d}{(n+1)}
\]
Finally, for any given integer $k$ we can express $k=d\lambda$ for
$\lambda=k/d$ and use the basic scaling property 
\[
\widehat{h}_{\lambda\phi}(\mathcal{X},\lambda\mathcal{L})=\lambda\widehat{h}_{\phi}(\mathcal{X},\mathcal{L}),
\]
to get
\[
2\widehat{h}_{\phi}(X_{a},\mathcal{O}(k))-2\widehat{h}_{F_{a}^{*}\phi}(X_{1},\mathcal{O}(k))=\frac{1}{d}\sum_{i}\log(|a_{i}|^{2})\frac{k}{(n+1)}
\]
 which concludes the proof. 
\end{proof}
\begin{lem}
\label{lem:scaling of integral}Given a metric $\phi$ on $\mathcal{O}(n+2-d)$
we have that 
\[
\log\int_{X_{a}}e^{-\phi}=\log\int_{X_{1}}e^{-F_{a}^{*}\phi}-d^{-1}\sum_{i}\log(|a_{i}|^{2}).
\]
\end{lem}

\begin{proof}
Let $X$ be the non-singular hypersurface of $\P^{n+1}$ cut out by
a given homogeneous polynomial $s$ in $\C^{n+2}.$ Let $A_{X}$ be
the the zero-locus of $s$ in $\C^{n+2}-\{0\}$ and assume that $ds\neq0$
on $A_{X}.$ Let $\Omega_{s}$ be the holomorphic top form on $A_{X}$
defined by the relation $\Omega_{s}\wedge ds=dz$ on $A_{X},$ where
$dz:=dz_{0}\wedge\cdots\wedge dz_{n+1}.$ We can identify $X$ with
$A_{X}/\C^{*}-$using the standard $\C^{*}-$action on $\C^{n+2}.$
Denote by $\delta$ interior multiplication with the holomorphic vector
field generating the $\C^{*}-$action. Assume now that $X$ is Fano.
A given metric $\phi$ on $-K_{X}$ then corresponds to a one-homogeneous
function $r$ on $A_{X}$ (using the adjunction isomorphism \ref{eq:adjunction isomo}
and by identifying $\C^{n+2}-\{0\}$ with the complement of the zero-section
in $\mathcal{O}(1)^{*}\rightarrow X).$ Moreover, lifting the adjunction
isomorphism \ref{eq:adjunction isomo} to $A_{X}$ yields the following
well-known formula (which applies in the general setup of Fano varieties
over local fields; cf. \cite[Lemma 4.2.2]{pey}):
\[
\int_{X}e^{-\phi}=c\int_{\{s=0\}/\C^{*}}r^{d-(n+2)}(\delta\Omega_{s}\wedge\overline{\delta\Omega_{s}})
\]
for a non-zero constant $c$ only depending on $n$ and $d$ (since
$\Omega_{s}$ has degree $(n+2-d)$ wrt the $\C^{*}-$action, the
$(n,n)-$form $r^{d-(n+2)}(\delta\Omega_{s}\wedge\overline{\delta\Omega_{s}})$
is $\C^{*}-$invariant and thus descends to a real top form on $X).$
Hence, setting $F:=F_{a},$
\[
\int_{X}e^{-\phi}=c\int_{\{F^{*}s=0\}/\C^{*}}(F^{*}r)^{d-(n+2)}(\delta F^{*}\Omega_{s}\wedge\overline{\delta F^{*}\Omega_{s}}).
\]
 To concude the proof it will thus be enough to verify that $F^{*}\Omega_{s}=\Omega_{F^{*}s}a_{0}^{-1/d}\cdots a_{n+1}^{-1/d}.$
To this end, note that applying $F^{*}$ to the defining relation
for $\Omega_{s}$ yields $F^{*}\Omega_{s}\wedge d(F^{*}s)=F^{*}dz.$
Since $F^{*}dz=a_{0}^{-1/d}\cdots a_{n+1}^{-1/d}dz$ this concludes
the proof.
\end{proof}
It follows directly from Proposition \ref{prop:red to fermat} that
\begin{equation}
h_{\text{can }}(\mathcal{X}_{a})=h_{\text{can }}(\mathcal{X}_{1})+(n+1)(n+2-d)^{n}d\left(\frac{n+2-d}{(n+1)}-1\right)d^{-1}\sum_{i}\log(|a_{i}|^{2}).\label{eq:canonical height of X a}
\end{equation}
 In particular, since $n+2-d\leq n+1,$ this means that $h_{\text{can }}(\mathcal{X}_{a})\leq h_{\text{can }}(\mathcal{X}_{1})$
and thus the proof of Theorem \ref{thm:diagonal hypersurface intro}
is reduced to the case of the Fermat hypersurface $\mathcal{X}_{1}.$
\begin{rem}
\label{rem:formula with T}More generally, let $\mathcal{X}$ be a
hypersurface in $\P_{\Z}^{n+1}$ cut-out by a homogeneous polynomial
$s$ of degree $d$ of the form $T^{*}s_{1}$ where $T\in GL(n+2,\C).$
Then formula \ref{eq:canonical height of X a} can be generalized
as follows (as shown in essentially the same manner as before): 
\[
h_{\text{can }}(\mathcal{X})=h_{\text{can }}(\mathcal{X}_{1})+(n+1)(n+2-d)^{n}d\left(\frac{n+2-d}{(n+1)}-1\right)\log(|\det T|^{2}).
\]
\end{rem}

\subsection{Reduction to log hyperplane arrangements }

Fix a degree $d(\leq n+1)$ and denote by $\mathcal{X}$ the corresponding
Fermat hypersurface. The Fermat hypersurface of degree one will be
denoted by $\mathcal{Y}.$ We will next express the canonical height
$h_{\text{can }}(\mathcal{X})$ in terms of the canonical height $h_{\text{can }}(\mathcal{Y},\mathcal{D})$
where $\mathcal{D}$ is the divisor on $\mathcal{Y}$ defined by 
\[
\mathcal{D}=(1-1/d)\left[x_{0}=0\right]+...+(1-1/d)\left[x_{n}=0\right]+(1-1/d)\left[\sum_{i=0}^{n}x_{i}=0\right]
\]
 where $x_{i}$ denotes the homogenous coordinates on $\P_{\Z}^{n+1}$
restricted to $\mathcal{Y}.$ 
\begin{prop}
\label{prop:red to log hyper}Denote by $\mathcal{X}$ the Fermat
hypersurface of a given degree $m(\leq n+1)$ and by $\mathcal{Y}$
the Fermat hypersurface of degree one, endowed with the divisor $\mathcal{D}.$
Then 
\begin{equation}
\hat{h}_{\text{can }}(\mathcal{X})=\hat{h}(\mathcal{Y},\mathcal{D})-\frac{1}{2}\log\frac{V(X)}{V(Y,\Delta)}\label{eq: height hypersurface to plane arrangement}
\end{equation}
\end{prop}

\begin{proof}
By the adjunction formula we have isomorphisms $-K_{\mathcal{X}}\simeq(n+2-m)\mathcal{O}(1)|_{\mathcal{X}}$
and $-\mathcal{K}_{(\mathcal{Y},\Delta)}\simeq(n+2-m)(1/m)\mathcal{O}(1)|_{\mathcal{Y}}.$
Consider the following morphism: 
\[
F:\,\C^{n+2}\rightarrow\C^{n\text{+2}},\,\,\,(x_{0},...,x_{n+1})\mapsto(x_{0}^{m},...,x_{n+1}^{m}),
\]
 which induces a map $\P^{n+1}\rightarrow\P^{n\text{+1}}$ and a lift
to $\mathcal{O}(1),$ which is naturally defined over $\Z$ and satisfies
$F^{*}\mathcal{O}(1)\simeq m\mathcal{O}(1).$ In particular, it induces
a morphism 

\[
F:\,\mathcal{X}\rightarrow\mathcal{Y},\,\,\,\,F^{*}(-\mathcal{K}_{(\mathcal{Y},\Delta)})\simeq-K_{\mathcal{X}}
\]
under the adjunction isomorphisms. We recall that, by basic functorial
properties of heights,
\begin{equation}
\hat{h}(\mathcal{X},F^{*}\mathcal{L})=\hat{h}(\mathcal{Y},\mathcal{L}).\label{eq:normalized height under morphism}
\end{equation}
In fact, in this case this formula follows directly from the analytic
representation of the height in Lemma \ref{lem:(Restriction-formula)-Let},
using that $F$ preserves the Weil metric $\phi_{0}.$ In particular,
setting $\mathcal{L}:=(n+2-m)(1/m)\mathcal{O}(1)_{|\mathcal{Y}}$
and using the adjunction isomorphisms yields $\hat{h}(\mathcal{X},-K_{\mathcal{X}},F^{*}\phi)=\hat{h}(\mathcal{Y},-\mathcal{K}_{(\mathcal{Y},\Delta)},\phi).$
Thus, all that remains is to show that 
\[
\int_{X}\mu_{F^{*}\phi}=m^{-(n+1)}\int_{Y}\mu_{\phi},
\]
where we have used that $F^{*}\phi$ induces a metric on $-K_{X}$
and $\phi$ induces a metric on $-K_{(Y,\Delta)}.$ Since $F$ has
topological degree $m^{(n+1)}$ we have $F_{*}[X_{m}]=m^{(n+1)}[Y]$
as homology classes and thus it will be enough to show that
\[
F^{*}\mu_{\phi}=m^{2(n+1)}\mu_{F^{*}\phi}.
\]
 To this end consider the affine piece $\C^{n+1}$ of $\P^{n+1}$
where $x_{0}\neq0.$ Setting $z_{i}=x_{i}/x_{0}$ for $i=1,...,n+1)$
we can, locally, parametrize $X$ by the coordinates $z_{1},...,z_{n}.$
In these coordinates a metric $\psi$ on the restriction of $\mathcal{O}(1)$
to $X$ induces, by the adjunction isomorphism, a metric on $-K_{X}$
and thus a measure on $X$ locally expressed as

\begin{equation}
\mu_{\psi}=\frac{e^{-(n+2-m)\psi}}{\left(m|z_{n+1}|^{m-1}\right)^{2}}\frac{i}{2}\mathrm{d}z_{1}\wedge\mathrm{d}\bar{z}_{1}\cdots\frac{i}{2}\mathrm{d}z_{n}\wedge\mathrm{d}\bar{z}_{n}.\label{eq:mu phi on quadric-1}
\end{equation}
 To see this, recall that, by definition, 
\[
\mu_{\psi}:=\left\Vert \mathrm{d}z_{1}\wedge\cdots\wedge\mathrm{d}z_{n}\right\Vert _{\psi}^{-2}\frac{i}{2}\mathrm{d}z_{1}\wedge\mathrm{d}\bar{z}_{1}\cdots\frac{i}{2}\mathrm{d}z_{n}\wedge\mathrm{d}\bar{z}_{n}.
\]
 In the affine piece $\C^{n+1}$ of $\P^{n+1}$ we can express
\[
s=fx_{0}^{\otimes2},\,\,\,\,f=1+\sum_{i=1}^{n}z_{i}^{m}+z_{n+1}^{m}
\]
 By the adjunction isomorphism \ref{eq:adjunction isomo} we have
\[
\left\Vert \mathrm{d}z_{1}\wedge\cdots\wedge\mathrm{d}z_{n}\right\Vert :=\left\Vert \mathrm{d}z_{1}\wedge\cdots\wedge\mathrm{d}z_{n}\wedge\mathrm{d}s\right\Vert =\left\Vert \mathrm{d}z_{1}\wedge\cdots\wedge\mathrm{d}z_{n}\wedge\mathrm{d}f\right\Vert \left\Vert x_{0}^{\otimes m}\right\Vert .
\]
 Since $\mathrm{d}z_{1}\wedge\cdots\wedge\mathrm{d}z_{n}\wedge\mathrm{d}f=\mathrm{d}z_{1}\wedge\cdots\wedge\mathrm{d}z_{n}\wedge\mathrm{d}z_{n+1}\partial f/\partial z_{n+1}$
this means that 
\[
\left\Vert \mathrm{d}z_{1}\wedge\cdots\wedge\mathrm{d}z_{n}\right\Vert ^{2}:=\left|\frac{\partial f}{\partial z_{n+1}}\right|^{2}\left\Vert \mathrm{d}z_{1}\wedge\cdots\wedge\mathrm{d}z_{n}\wedge\mathrm{d}z_{n+1}\right\Vert \left\Vert x_{0}^{\otimes m}\right\Vert =e^{(n+2)\psi}e^{-m\psi},
\]
giving 
\begin{equation}
\mu_{\phi}=\frac{e^{-(n+2-m)\psi}}{|\partial f/\partial z_{n+1}|^{2}}\frac{i}{2}\mathrm{d}z_{1}\wedge\mathrm{d}\bar{z}_{1}\cdots\frac{i}{2}\mathrm{d}z_{n}\wedge\mathrm{d}\bar{z}_{n},\label{eq: representation of muphi}
\end{equation}
which proves \ref{eq:mu phi on quadric-1}. Likewise, we can parametrize
the affine piece of $Y$ by the coordinates $z_{1},...,z_{n}.$ A
given metric $\phi$ on the restriction of $\mathcal{O}(1)$ to $Y$
induces, by the adjunction isomorphism a measure (defined with respect
to the divisor $\mathcal{D})$
\[
\mu_{\phi}=e^{-(n+2-m)m^{-1}\phi}\left|z_{n+1}\right|^{-2(1-1/m)}\left|z_{1}\right|^{-2(1-1/m)}\cdots\left|z_{n}\right|^{-2(1-1/m)}\frac{i}{2}\mathrm{d}z_{1}\wedge\mathrm{d}\bar{z}_{1}\cdots\frac{i}{2}\mathrm{d}z_{n}\wedge\mathrm{d}\bar{z}_{n}.
\]
Since $F^{*}z_{i}=z_{i}^{m}$ this means that
\[
F^{*}\mu_{\phi}=e^{-(n+2-m)f^{*}(m^{-1}\phi)}\left|z_{n+1}\right|^{-2(m-1)}\left|z_{1}\right|^{-2(m-1)}\cdots\left|z_{n}\right|^{-2(m-1)}\frac{i}{2}\mathrm{d}(z_{1}^{m})\wedge\mathrm{d}(\bar{z}_{1}^{m})\cdots\frac{i}{2}\mathrm{d}(z_{n}^{m})\wedge\mathrm{d}(\bar{z}_{n}^{m}).
\]
 Finally, since $\mathrm{d}(z^{m})=mz^{m-1}$ this proves the desired
identity \ref{eq: height hypersurface to plane arrangement}, using
the representation \ref{eq: representation of muphi} with $\psi=F^{*}(m^{-1}\phi).$
\end{proof}

\subsection{Conclusion of the proof of Theorem \ref{thm:diagonal hypersurface intro} }

The affine projection $(x_{0},..,x_{n+1})\mapsto(x_{0},...,x_{n})$
induces an isomorphism from $\mathcal{Y}$ to $\P_{\Z}^{n},$ identifying
$(\mathcal{Y},\mathcal{D})$ with a hyperplane arrangement $(\P_{\Z}^{n},\mathcal{D})$
with simple normal crossings. It follows readily from the definition
of $\mathcal{D}$ and the criterion \ref{eq:def of C} that $(\P_{\Z}^{n},\mathcal{D})$
is K-semistable. Hence, combining Proposition \ref{prop:red to log hyper}
with refined bound following the statement of Theorem \ref{thm:hyperplane intro}
yields 
\[
\hat{h}_{\text{can }}(\mathcal{X})\leq\hat{h}_{\text{can }}(\P_{\Z}^{n},\mathcal{D}_{t})-\frac{1}{2}\log\frac{V(X)}{V(\P^{n},\Delta_{t})},
\]
 where $\mathcal{D}_{t}$ is the toric divisor on $\P_{\Z}^{n}$ such
that $(\P_{\Z}^{n},\mathcal{D}_{t})$ is K-semistable and $V(\P^{n},\Delta_{t})=V(\P^{n},\Delta).$
The explicit formula for $\hat{h}_{\text{can }}(\P_{\Z}^{n},\mathcal{D}_{t})$
thus yields 
\[
\hat{h}_{\text{can }}(\mathcal{X})\leq\hat{h}_{\text{can }}(\P_{\Z}^{n})-\frac{1}{2}\log\frac{V(X)}{V(\P^{n})}.
\]
 Multiplying both sides with $V(X)$ reveals that 
\[
h_{\text{can }}(\mathcal{X})\leq\lambda h_{\text{can }}(\P_{\Z}^{n})-\frac{1}{2}V(X)\log\lambda,\,\,\,\lambda:=V(X)/V(\P^{n}).
\]
 Since $\lambda\in]0,1[$ it thus follows from Lemma \ref{lem:scaled polytope}
that the right hand side above is increasing with respect to $\lambda$
and thus maximal when $\lambda=1,$ giving $h_{\text{can }}(\mathcal{X},\mathcal{D})\leq h_{\text{can }}(\P_{\Z}^{n}).$
Moreover, the equality is strict if $d\geq2$ since then $\lambda<1.$ 

\section{\label{sec:Arithmetic-surfaces}Arithmetic log surfaces }

In this section we will prove Theorem \ref{thm:Mab on log surface}.
Throughout the section $\mathcal{X}$ will be assumed normal. Given
an effective divisor $\Delta_{\Q}$ on $\P_{\Q}^{1}$ supported at
three points, such that $-K_{(\P^{1},\Delta)}$ is ample, we define
the \emph{canonical model} of $(\P_{\Q}^{1},\Delta_{\Q};-K_{(\P_{\Q}^{1},\Delta_{\Q})})$
over $\Z$ as $(\P_{\Z}^{1},\mathcal{D}_{c};-K_{(\P_{\Z}^{1},\mathcal{D}_{c})})$
where $\mathcal{D}_{c}$ is the Zariski closure of the divisor on
$\P_{\Q}^{1}$ supported on $\{0,1,\infty\},$ having the same coefficents
as $\Delta_{\Q}.$ That this is a model over $\Z$ follows from the
basic fact that any three points on $\P_{\Q}^{1}$ can be mapped to
to $\{0,1,\infty\},$ by an automorphism of $\P_{\Q}^{1},$ 

As explained in Section \ref{subsec:Arithmetic-log-surfaces}, by
Theorem \ref{thm:hyperplane intro}, the proof is reduced to showing
that for any fixed metric on $-K_{(\P^{1},\Delta)}:$
\begin{itemize}
\item The canonical integral model $(\P_{\Z}^{1},\mathcal{D}_{c};-\mathcal{K}_{(\P_{\Z}^{1},\mathcal{D}_{c})})$
of $(\P_{\Q}^{1},\Delta_{\Q};-K_{(\P_{\Q}^{1},\Delta_{\Q})})$ minimizes
$\mathcal{M}_{(\mathcal{X},\mathcal{D})}(\overline{\mathcal{L}})$
over all integral models $(\mathcal{X},\mathcal{D};\mathcal{L})$
of $(\P_{\Q}^{1},\Delta_{\Q};-K_{(\P_{\Q}^{1},\Delta_{\Q})}).$
\item When $\Delta_{\Q}=0$ the minimum is uniquely attained for $\mathcal{X}=\P_{\Z}^{1},$
up to isomorphisms over $\Z.$
\end{itemize}

\subsection{Preliminaries on log canonical thresholds}

Following \cite{ko,ta} a log pair $(\mathcal{X},\mathcal{D})$ is
said to be \emph{log canonical (lc)} if for any normal blow-up morphism
$p:\mathcal{Y}\rightarrow\mathcal{X}$ 
\[
\mathcal{K}_{\mathcal{Y}/\mathcal{X}}-p^{*}\mathcal{D}=\sum_{i}a_{i}E_{i},\,\,\,a_{i}\geq-1,\,\,\,\mathcal{K}_{\mathcal{Y}/\mathcal{X}}:=\mathcal{K}_{\mathcal{Y}}-p^{*}\mathcal{K}_{\mathcal{X}}
\]
 where the prime divisor $E_{i}$ is either an exceptional divisor
of $p$ or the proper transform of a component of $\mathcal{D}.$
The\emph{ log canonical threshold} of a $\Q-$divisor $F$ on $\mathcal{X}$
with respect to the log pair $(\mathcal{X},\mathcal{D})$ is defined
by
\[
\text{lct \ensuremath{(\mathcal{X},\mathcal{D};F):=\sup_{t>0}\left\{ t:\,(\mathcal{X},tF+\mathcal{D})\,\text{is lc}\right\} }}
\]
The following lemma follows readily from the definition: 
\begin{lem}
\label{Lemma:lct as inf}For any normal blow-up morphism $p:\mathcal{Y}\rightarrow\mathcal{X}$ 

\[
\text{lct \ensuremath{(\mathcal{X},\mathcal{D};F)\leq\inf_{i}\frac{a_{i}-b_{i}+1}{c_{i}}}}
\]
where $a_{i},b_{i}$ and $c_{i}$ denote the order of vanishing along
the $p-$exceptional prime divisor $E_{i}$ of $\mathcal{K}_{\mathcal{Y}/\mathcal{X}},$
$p^{*}\mathcal{D}$ and $p^{*}F,$ respectively and $i$ ranges over
all $p-$exceptional prime divisors. 
\end{lem}

\subsection{Preparations for the proof of Theorem \ref{thm:Mab on log surface}}

The following result is a logarithmic generalization of \cite[Thm 2.14 (3)]{o}
(in the case of arithmetic surfaces).
\begin{lem}
\label{lem:Odaka log alpha }Let $(X,D)$ be a log Fano curve over
$\C$ and $(\mathcal{X},\mathcal{D})$ an arithmetic log Fano model
for $(X,D)$ such the fibers $\mathcal{X}_{b}$ of $\mathcal{X}$
are reduced and irreducible and the divisor $\mathcal{D}$ is horizontal
(i.e. $\mathcal{D}$ is the Zariski closure of $D)$. Assume that
\[
\alpha(\mathcal{X},\mathcal{D}):=\inf_{b,F}\text{lct \ensuremath{(\mathcal{X},\mathcal{D}+\mathcal{X}_{b};\,F)}}\geq1/2
\]
 where the inf runs over all effective $\Q-$divisors $F$ on $\mathcal{X}$
linearly equivalent to $-\mathcal{K}_{(\mathcal{X},\mathcal{D})}$
and closed points $b$ in the base $\mathcal{B}:=\text{Spec \ensuremath{\Z} }$
such that $F$ does not contain the support of $\mathcal{X}_{b}.$
Then 
\[
\frac{1}{2}\overline{\mathcal{L}}^{2}+\overline{\mathcal{K}}_{(\mathcal{X}',\mathcal{D})}\cdot\overline{\mathcal{L}}\geq\frac{1}{2}\overline{\mathcal{L}}^{2}+\overline{\mathcal{K}}_{(\mathcal{X},\mathcal{D})}\cdot\overline{\mathcal{L}}
\]
 for any relatively ample model $(\mathcal{X}',\mathcal{D}';\mathcal{L}')$
of $(X,D;-K_{(X,D)})$ and given metrics on $-K_{(X,D)}$ and $L.$
\end{lem}

In the proof we fix once and for all metrics on $-K_{(X,D)}$ and
$L$ and set
\begin{equation}
\mathcal{M}_{(\mathcal{X},\mathcal{D})}(\mathcal{L}):=\frac{1}{2}\overline{\mathcal{L}}^{2}+\overline{\mathcal{K}}_{(\mathcal{X},\mathcal{D})}\cdot\overline{\mathcal{L}}\label{eq:more general log Mab}
\end{equation}
for the corresponding metrized lines bundles $\overline{\mathcal{L}}$
and $\overline{\mathcal{K}}_{(\mathcal{X},\mathcal{D})}.$ Thus $\mathcal{M}_{(\mathcal{X},\mathcal{D})}(\mathcal{L})$
specializes to the arithmetic log Mabuchi functional $\mathcal{M}_{(\mathcal{X},\mathcal{D})}(\overline{\mathcal{L}})$
(formula \ref{eq:def of log Mab surface intro}) precisely when the
metric on $\overline{\mathcal{K}}_{(\mathcal{X},\mathcal{D})}$ is
the one induced from the curvature form $\omega$ of $\overline{\mathcal{L}}.$
But here it will be convenient to consider the present more general
setup.

\subsubsection*{Proof when $D$ is trivial }

To fix ideas we first consider the case when $D$ is trivial. Set
$\mathcal{B}:=\text{Spec \ensuremath{\Z} }$ and $\mathcal{L}:=-\mathcal{K}_{\mathcal{X}}.$
To simplify the notation we will remove the bar indicating the metric
in the notation for the arithmetic intersection numbers. Anyhow, all
the arithmetic intersections will be computed over the closed points
$b$ in the base $\mathcal{B}$ and are thus independent of the choice
of metric (since they are proportional to the algebraic intersections
on the scheme $\pi^{-1}(b)$ over the residue field of $b$). If $F_{1}$
and $F_{2}$ are $\Q-$divisors we will write $F_{1}\geq F_{2}$ if
$F_{1}-F_{2}$ is effective.

\emph{Step 1:} \emph{It is enough to consider the case of a relatively
semi-ample model of the form $(\mathcal{X}',\mathcal{L}')=(\mathcal{Y},p^{*}\mathcal{L}-E)$
where $p:\mathcal{Y}\rightarrow\mathcal{X}$ is the blow-up along
a closed subscheme $\mathcal{Z}$ of $\mathcal{X}$ and $E$ is an
effective $p-$exceptional divisor on $\mathcal{Y}$ whose support
contains all the $p-$exceptional prime divisors and such that for
any $b\in\pi(\mathcal{Z})$ $p^{*}\mathcal{L}-E$ admits a global
section $s_{b}$ not vanishing identically along $\mathcal{Y}_{b}.$}

This is shown precisely as in the proof of \cite[Prop 3.10]{od0}
- for completeness a proof is given in Step 1 in Section \ref{subsec:Conclusion-of-the proof log surface}
below. 

\emph{Step2: The inequality holds in the case of Step 1.}

First observe that
\begin{equation}
\mathcal{M}_{\mathcal{Y}}(\mathcal{L}')-\mathcal{M}_{\mathcal{X}}(\mathcal{L})=\mathcal{L}'\cdot\left(\mathcal{K}_{\mathcal{Y}/\mathcal{X}}-\frac{1}{2}E\right).\label{eq:difference of Mab}
\end{equation}
Indeed, rewriting 
\[
\mathcal{M}_{\mathcal{X}}(\mathcal{L})=-\frac{\mathcal{L}^{2}}{2}+\cdot\mathcal{L}\cdot(\mathcal{L}+\mathcal{K}_{\mathcal{X}})
\]
(and likewise for $(\mathcal{Y},\mathcal{L}')$) the left hand side
in formula \ref{eq:difference of Mab} may be expressed as
\[
\frac{p^{*}\mathcal{L}^{2}-\mathcal{L}'^{2}}{2}+\mathcal{L}'\cdot(\mathcal{K}_{\mathcal{Y}/\mathcal{X}}-E)=\frac{E\cdot(p^{*}\mathcal{L}+\mathcal{L}')}{2}+\mathcal{L}'\cdot(\mathcal{K}_{\mathcal{Y}/\mathcal{X}}-E).
\]
 Since $E\cdot p^{*}\mathcal{L}=0$ this proves formula \ref{eq:difference of Mab}.

Since $\mathcal{L}'$ is relatively semi-ample it will thus be enough
to show that the vertical exceptional divisor $\mathcal{K}_{\mathcal{Y}/\mathcal{X}}-\frac{1}{2}E$
is effective. This means, by the assumption on $\alpha(\mathcal{X})(:=\alpha(\mathcal{X},0)),$
that it is enough to show that 
\begin{equation}
\mathcal{K}_{\mathcal{Y}/\mathcal{X}}-\alpha(\mathcal{X})E\geq0\label{eq:K-alphaE is nonneg}
\end{equation}
To fix ideas first assume that $\pi(\mathcal{Z})$ is supported on
a single point that we denote by $b.$ By Step 1, we can express $s_{b}=p^{*}s$
for a global section $s$ of $\mathcal{L}\rightarrow\mathcal{X}$
whose zero-divisor $F$ does not vanish identically on $\mathcal{X}_{b}$
and such that $p^{*}F-E$ is effective. Since $F$ is a contender
for the inf defining $\alpha(\mathcal{X})$ we have 

\[
\alpha(\mathcal{X})\leq\text{lct}(\mathcal{X},\mathcal{X}_{b};\,F),
\]
Next, since $p^{*}F\geq E$ it follows from Lemma \ref{Lemma:lct as inf}
that 
\[
\text{lct}(\mathcal{X},\mathcal{X}_{b};\,F)\leq\inf_{i}\frac{a_{i}+1-b_{i}}{c_{i}}
\]
 where $i$ runs over the $p-$exceptional irreducible prime divisors
$E_{i}$ of $\mathcal{Y}$ and $a_{i},b_{i}$ and $c_{i}$ denote
the order of vanishing along $E_{i}$ of $\mathcal{K}_{\mathcal{Y}/\mathcal{X}},$
$\mathcal{Y}_{b}$ and $E$ respectively. Note that $c_{i}>0$ (since
the support of $E$ contains the support of all $p-$exceptional divisors)
and $b_{i}\geq1$ (since $\mathcal{Z}$ is assumed to be supported
in $\mathcal{X}_{b}).$ Thus
\[
\alpha(\mathcal{X})\leq\inf_{i}\frac{a_{i}+1-b_{i}}{c_{i}}\leq\inf_{i}\frac{a_{i}}{c_{i}},
\]
giving
\[
\mathcal{K}_{\mathcal{Y}/\mathcal{X}}-\alpha(\mathcal{X})E\geq\sum_{i}a_{i}E_{i}-\left(\min_{j}\frac{a_{j}}{c_{j}}\right)E_{i}=\sum_{i}\left(\frac{a_{i}}{c_{i}}-\left(\min_{j}\frac{a_{j}}{c_{j}}\right)\right)c_{i}E_{i}\geq0,
\]
 which proves \ref{eq:K-alphaE is nonneg}. Finally, consider the
general case when the support of $\pi(\mathcal{Z})$ consists of a
finite number of points $b_{m}$ in $\mathcal{B}.$ We then split
the vertical divisors $\mathcal{K}_{\mathcal{Y}/\mathcal{X}}$ and
$E$ into the components $\mathcal{K}_{\mathcal{Y}/\mathcal{X}}^{(m)}$
and $E^{(m)}$ over $b_{m}:$
\[
\mathcal{K}_{\mathcal{Y}/\mathcal{X}}-\alpha(\mathcal{X})E=\sum_{m}\mathcal{K}_{\mathcal{Y}/\mathcal{X}}^{(m)}-\alpha(\mathcal{X})E^{(m)}
\]
 and apply the previous bound for each fixed $m$ (with $b$ replaced
by $b_{m})$ to get that $\mathcal{K}_{\mathcal{Y}/\mathcal{X}}^{(m)}-\alpha(\mathcal{X})E^{(m)}\geq0$
and thus $\mathcal{K}_{\mathcal{Y}/\mathcal{X}}-\alpha(\mathcal{X})E\geq0,$
as desired. 

\subsubsection*{Proof for log pairs }

Just as in the previous case it is enough to consider the special
case of Step 2. In this case formula \ref{eq:difference of Mab} readily
generalizes to
\begin{equation}
\mathcal{M}_{(\mathcal{Y},q^{*}\mathcal{D}')}(\mathcal{L}')-\mathcal{M}_{(\mathcal{X},\mathcal{D})}(\mathcal{L})=\mathcal{L}'\cdot\left(\mathcal{D}'-p^{*}\mathcal{D}+\mathcal{K}_{\mathcal{Y}/\mathcal{X}}-\frac{1}{2}E\right).\label{eq:difference Mab}
\end{equation}
 As before it will thus be enough to show that
\begin{equation}
\mathcal{K}_{\mathcal{Y}/\mathcal{X}}+q^{*}\mathcal{D}'-p^{*}\mathcal{D}-\alpha(\mathcal{X},\mathcal{D})E\geq0.\label{eq:log K minus alpha E}
\end{equation}
 To simplify the exposition we will assume that $\pi(\mathcal{Z})$
is a single closed point in $\mathcal{B},$ denoted by $b$ (the general
case is shown in a similar way by decomposing wrt the components of
$\pi(\mathcal{Z})$ as above). By the definition of $\alpha(\mathcal{X},\mathcal{D})$
\begin{equation}
\alpha(\mathcal{X},\mathcal{D})\leq\text{lct}(\mathcal{X},\mathcal{D}+\mathcal{X}_{b};\,F),\label{eq:upper bound on log alpha}
\end{equation}
Next, since $p^{*}F-E$ is effective, i.e. $p^{*}F\geq E$ Lemma \ref{Lemma:lct as inf}
yields
\[
\text{lct}(\mathcal{X},\mathcal{D}+\mathcal{X}_{b};\,F)\leq\inf_{i}\frac{a_{i}+1-d_{i}-b_{i}}{c_{i}},
\]
where $a_{i},b_{i},$ $c_{i}$ and $d_{i}$ are the order of vanishing
along $E_{i}$ of $\mathcal{K}_{\mathcal{Y}/\mathcal{X}},$ $\mathcal{Y}_{b}$,
$E$ and $p^{*}\mathcal{D}$ respectively. In particular, $b_{i}\geq1$
since $\mathcal{Z}$ is supported in $\mathcal{X}_{b},$ Hence,
\[
\alpha(\mathcal{X},\mathcal{D})\leq\inf_{i}\frac{a_{i}+1-d_{i}-b_{i}}{c_{i}}.
\]
Next, we may decompose 
\[
p^{*}\mathcal{D}=\left(p^{*}\mathcal{D}\right)_{\text{hor}}+(p^{*}\mathcal{D})_{\text{ex}},
\]
 where $\left(p^{*}\mathcal{D}\right)_{\text{hor}}$ is the horizontal
divisor obtained as the proper transform of the horizontal divisor
$\mathcal{D}$ and $(p^{*}\mathcal{D})_{\text{ex}}$ is $p-$exceptional.
By \ref{eq:upper bound on log alpha} $\mathcal{K}_{\mathcal{Y}/\mathcal{X}}-(p^{*}\mathcal{D})_{\text{ex}}-\alpha E_{\text{ex}}\geq$
\[
\geq\sum_{i}(a_{i}-d_{i})E_{i}-\left(\min_{j}\frac{a_{j}-b_{j}+1-d_{j}}{c_{j}}\right)c_{i}E_{i}\geq\sum_{i}\left(\frac{(a_{i}-d_{i})}{c_{i}}E_{i}-\left(\min_{j}\frac{a_{j}-d_{j}}{c_{j}}\right)\right)c_{i}E_{i}\geq0
\]
 using that $b_{j}\geq1.$ Hence, 
\[
\mathcal{K}_{\mathcal{Y}/\mathcal{X}}+\mathcal{D}'-p^{*}\mathcal{D}-\alpha(\mathcal{X},\mathcal{D})E_{\text{ex}}\geq\mathcal{D}'-\left(p^{*}\mathcal{D}\right)_{\text{hor}}.
\]
 But, since both $(\mathcal{X}',\mathcal{D}')$ and $(\mathcal{X},\mathcal{D})$
are models for $(X,\Delta)$ and $\mathcal{D}$ is assumed horizontal
it follows that $q^{*}\mathcal{D}'-\left(p^{*}\mathcal{D}\right)_{\text{hor}}$
is an effective vertical divisor and hence 
\[
q^{*}\mathcal{D}'-\left(p^{*}\mathcal{D}\right)_{\text{hor}}\geq0,
\]
 which concludes the proof of the inequality \ref{eq:log K minus alpha E}. 
\begin{lem}
\label{lem:alpha geq one half}If $(X,\Delta)$ is a K-semistable
log Fano curve over $\C,$ then the canonical model $(\P_{\Z}^{1},\mathcal{D}_{c})$
of $(X,\Delta)$ satisfies
\[
\alpha(\mathcal{X},\mathcal{D})\geq1/2
\]
 and the inequality is strict if $(X,\Delta)$ is K-stable.
\end{lem}

\begin{proof}
By inversion of adjunction on surfaces over excellent schemes \cite{ta}
\[
\text{\text{lct \ensuremath{(\mathcal{X},\mathcal{D}+\mathcal{X}_{b};\,F)}}}=\text{lct \ensuremath{(\mathcal{X},\mathcal{D}_{|\mathcal{X}_{b}};F_{|\mathcal{X}_{b}})}},
\]
 if $F$ does not contain the support of the divisor $\mathcal{X}_{b}.$
In the present case $\mathcal{X}_{b}=\P_{\F_{b}}^{1},$ where $b$
has been identified with a prime number and $\F_{b}$ denotes the
field with $b$ elements. Decomposing 
\[
\mathcal{D}_{c}=\sum w_{i}\mathcal{D}_{i}
\]
the K-semistability assumption is, by \ref{eq:weight cond}, equivalent
to the condition 
\begin{equation}
w_{j}\leq\frac{1}{2}\sum_{i}w_{i},\,\,\,\forall j.\label{eq:weight cond in pf log alpha geq half}
\end{equation}
We recall that for any curve $C$ over a perfect field (here taken
to be $\P_{\F_{b}}^{1}$) an effective $\Q-$divisor $F$ on $C$
is lc iff all its coefficients are less then are equal to one \cite{ko,ta}.
Since $-K_{\P_{\F_{b}}^{1}}$ is linearly equivalent to $\mathcal{O}(2)$
and $D_{i|\mathcal{X}_{b}}$ is linearly equivalent to $\mathcal{O}(1)$
it thus follows from the weight condition \ref{eq:weight cond in pf log alpha geq half}
that $\text{lct}\ensuremath{(\mathcal{X},\mathcal{D}_{|\mathcal{X}_{b}};F_{|\mathcal{X}_{b}})}\geq1/2.$
Indeed, it is enough to consider the case when $F=(2-\sum_{i}w_{i})[x],$
where $[x]$ denotes the prime divisor on $\P_{\F_{b}}^{1}$ corresponding
to a closed point $x$ in $\P_{\F_{b}}^{1}.$ Then 
\[
\frac{1}{2}F+\mathcal{D}_{|\mathcal{X}_{b}}=(1-\frac{1}{2}\sum_{i}w_{i})[x]+\sum w_{i}[x_{i}].
\]
 The definition of $\mathcal{D}_{c}$ ensures that $[x_{i}]=[x_{j}]$
iff $i=j.$ In the case that $x\neq x_{i}$ for any $i$ the coefficients
of $F/2+\mathcal{D}_{|\mathcal{X}_{b}}$ are indeed less than or equal
to $1,$ since $w_{i}\in[0,1].$ Moreover, if $x=x_{j}$ then the
coefficient of index $j$ equals $(1-\frac{1}{2}\sum_{i}w_{i})+w_{j}$
which is less than are equal to $1,$ by the weight condition \ref{eq:weight cond in pf log alpha geq half}.

We will also need the following lemma, shown precisely as in the case
when $\mathcal{D}=0$ considered in \cite[Prop 5.3]{ab}. 
\end{proof}
\begin{lem}
\label{lem:inf is inf}Let $(\mathcal{X},\mathcal{D};\mathcal{L})$
be a polarized arithmetic log surfaces $(\mathcal{X},\mathcal{D};\mathcal{L})$
such that the complexification $(X,\Delta)$ of $(\mathcal{X},\mathcal{D})$
is a Fano variety and $\mathcal{L}\otimes\C=-K_{(X,\Delta)}.$ A metric
realizes the infimum 
\[
\inf_{\left\Vert \cdot\right\Vert }\mathcal{M}_{(\mathcal{X},\mathcal{D})}(\mathcal{L},\left\Vert \cdot\right\Vert )
\]
 over all locally bounded metrics on $-(K_{(X,D)})$ with positive
curvature current iff the metric is a log Kähler-Einstein metric.
In particular, in the case when $D=0$ any minimizer coincides with
the Fubini-Study metric up to the application of an automorphism of
$X$ and a scaling of the metric. Moreover, 

\[
\inf_{\left\Vert \cdot\right\Vert }\mathcal{M}_{(\P_{\Z}^{1},\mathcal{D})}\left(-\mathcal{K}_{(\P_{\Z}^{1},\mathcal{D})},\left\Vert \cdot\right\Vert \right)=-\sup_{\left\Vert \cdot\right\Vert }\frac{1}{2}\left(-\mathcal{K}_{(\P_{\Z}^{1},\mathcal{D})},\left\Vert \cdot\right\Vert \right)^{2}
\]
 where the sup in the right hand side is restricted to volume-normalized
metrics.
\end{lem}

\subsection{\label{subsec:Conclusion-of-the proof log surface}Conclusion of
the proof of Theorem \ref{thm:main log toric intro} and Corollary
\ref{cor:Conjecture holds for log Fano surfaces}}

Combining the previous first two lemmas immediately yields 
\begin{equation}
\mathcal{M}_{(\mathcal{X}',\mathcal{D}')}(\overline{\mathcal{L}'})\geq\mathcal{M}_{(\P_{\Z}^{1},\mathcal{D})}(\overline{-\mathcal{K}_{(\P_{\Z}^{1},\mathcal{D})}}).\label{eq:ineq log Mab in pf optimal}
\end{equation}
Applying the third lemma above thus gives 
\[
\mathcal{M}_{(\mathcal{X}',\mathcal{D}')}(\overline{\mathcal{L}'})\geq-\sup_{\left\Vert \cdot\right\Vert }\left(\frac{1}{2}\left(-\mathcal{K}_{(\P_{\Z}^{1},\mathcal{D})},\left\Vert \cdot\right\Vert \right)^{2}\right)
\]
 where the infimum in the left hand side is restricted to volume-normalized
metrics. Invoking Theorem \ref{thm:hyperplane intro} and using that
the Fubini-Study metric is a minimizer when $\mathcal{D}=0$ (by Lemma
\ref{lem:inf is inf}) thus proves the inequality in Theorem \ref{thm:Mab on log surface}
and it corollary. Moreover, Theorem \ref{thm:hyperplane intro} implies
that the inequality is strict, as soon as $D$ is non-trivial. 

\subsubsection{\label{subsec:The-equality-case}The equality case}

Consider now the case of equality in Theorem \ref{thm:Mab on log surface}
(and, as a consequence, $D=0):$
\begin{equation}
\mathcal{M}_{(\mathcal{X}',\mathcal{D}')}(\overline{\mathcal{L}'})=\mathcal{M}_{(\P_{\Z}^{1},0)}(\overline{-\mathcal{K}_{\P_{\Z}^{1}}})\label{eq:equality Mab in pf of equality case}
\end{equation}
where, in the rhs, $-\mathcal{K}_{\P_{\Z}^{1}}$ is endowed with the
Fubini-Study metric. By the minimizing property in Lemma \ref{Lemma:lct as inf},
when $D=0,$ the metric on $\mathcal{L}$ coincides with the Fubini-Study
metric up to the application of an automorphism of $X$ and scaling
of the metric. All that remains is to show is thus that $(\mathcal{X}',\mathcal{D}')$
is isomorphic to $(\P_{\Z}^{1},0).$ To this end first note that since
$\mathcal{X}(=\P_{\Z}^{1})$ and $\mathcal{X}'$ have the same generic
fiber they are birationally equivalent. Thus, there exists a normal
variety $\mathcal{Y},$ which is flat and projective over $\mathcal{B},$
dominating both $\mathcal{X}$ and $\mathcal{X}',$ with birational
morphisms 
\begin{equation}
p:\mathcal{Y}\rightarrow\mathcal{X},\,\,\,q:\mathcal{Y}\rightarrow\mathcal{X}'\label{eq:morphisms p and q}
\end{equation}
which are the identity over the generic point in $\mathcal{B}$ (a
concrete construction is given in Step 1 below). It will thus be enough
to show that the equality \ref{eq:equality Mab in pf of equality case}
implies that $p$ can be taken to be an isomorphism. Indeed, if $p$
is an isomorphism we get a birational morphism $q$ from $\P_{\Z}^{1}$
to $\mathcal{X}'$ and any such morphism is an isomorphism (since
the fibers of $\P_{\Z}^{1}$ over $\mathcal{B}$ are all reduced and
irreducible). Moreover, when $\mathcal{X}'$ is equal to $\P_{\Z}^{1}$
any $\mathcal{L}$ whose complexification equals $-K_{\P^{1}}$ is
isomorphic to $-\mathcal{K}_{\P_{\Z}^{1}}$ and the components of
any divisor $\mathcal{D}'$ on $\mathcal{X}'$ whose complexification
is trivial are fibers $\mathcal{X}_{b_{i}}$ of $\P_{\Z}^{1}$ (using
again that the fibers of $\P_{\Z}^{1}$ over $\mathcal{B}$ are all
reduced and irreducible). Hence, the assumed equality \ref{eq:equality Mab in pf of equality case}
implies, since $\mathcal{X}_{b_{i}}^{2}=0$ and $-\mathcal{K}_{\P_{\Z}^{1}}$
is relatively ample that $\mathcal{D}'$ is trivial, i.e. $\mathcal{D}'=0.$

Thus all that remains is to show that the assumed equality in formula
\ref{eq:equality Mab in pf of equality case} implies that the morphism
$p:\mathcal{Y}\rightarrow\mathcal{X}$ (in formula \ref{eq:morphisms p and q})
can be taken to be an isomorphism. 

\emph{Step 1: In the case of arithmetic surfaces $p:\,\mathcal{Y}\rightarrow\mathcal{X}$
can be taken as the successive blow-ups of $\mathcal{X}$ along a
finite number of closed points $x_{i}$ in regular surfaces $\mathcal{X}_{i}$
and there exists a $p-$exceptional and $p-$ample effective divisor
$E$ on $\mathcal{Y}$ and a morphism $q$ from }$\mathcal{Y}$ to
$\mathcal{X}$ \emph{such that $q^{*}\mathcal{L}'=p^{*}\mathcal{L}-E.$
In particular, $\mathcal{M}_{\mathcal{X}'}(\mathcal{L}')=\mathcal{M}_{\mathcal{Y}}(p^{*}\mathcal{L}-E).$}

This is shown as in the beginning of the proof of \cite[Prop 3.10]{od0},
as next explained. First note that, since $\mathcal{X}$ and $\mathcal{X}'$
have the same generic fiber, they are birationally equivalent. Since
$\mathcal{X}$ is normal this means that there exists a morphism $h:U\rightarrow\mathcal{X}'$
from a Zariski open subset $U$ in $\mathcal{X}$ of codimension two.
As a consequence, $h^{*}\mathcal{L}'$ extends to a $\Q-$line bundle
$\mathcal{L}''$ on $\mathcal{X}$ coinciding with $-\mathcal{K}_{\mathcal{X}}$
on the generic fiber. Since $\mathcal{X}=\P_{\Z}^{1}$ this implies
that $\mathcal{L}''$ is isomorphic to $-\mathcal{K}_{\mathcal{X}}$
(using that $\pi:\mathcal{X}\rightarrow\text{Spec \ensuremath{\Z}}$
has reduced irreducible fibers). Now fix a positive integer $k$ such
that $k\mathcal{L}'$ is a relatively very ample line bundle and take
a basis $s'_{i}$ in the free $\Z-$module $H^{0}(\mathcal{X}',k\mathcal{L}'$).
Then $s_{i}:=h^{*}s_{i}$ extends, since $\mathcal{X}$ is normal,
to a unique element in $H^{0}(\mathcal{X},k\mathcal{L}).$ Denote
by $\mathcal{J}$ the ideal sheaf on $\mathcal{X}$ generated by the
sections $s_{i}.$ Since $\mathcal{X}$ is a regular surface we get
after successive blow-ups (as stated in Step 1) a morphism $p:\mathcal{Y}\rightarrow\mathcal{X}$
from a regular surface $\mathcal{Y}$ to $\mathcal{X}$ with the property
that $p^{*}\mathcal{J}$ defines an effective $p-$exceptional divisor
$E_{k}$ on $\mathcal{Y}$ (using that $\Z$ is an excellent ring)
\cite{st}. Set 
\[
E:=k^{-1}E_{k}\,\,\,\,(E_{k}:=p^{*}\mathcal{J})
\]
By construction, $E_{k}$ is $p-$ample, 
\begin{equation}
H^{0}(\mathcal{Y},kp^{*}\mathcal{L}-E_{k})\cong H^{0}(\mathcal{X},k\mathcal{L}\otimes\mathcal{J})\cong H^{0}(\mathcal{X}',k\mathcal{L}')\label{eq:isom}
\end{equation}
 and the global sections of $kp^{*}\mathcal{L}-E_{k}$ induce a morphism
$q$ to $\mathcal{X}'$ such that $q^{*}\mathcal{L}'=p^{*}\mathcal{L}-E.$
Finally, note that
\[
\mathcal{M}_{\mathcal{X}'}(\mathcal{L}')=\mathcal{M}_{\mathcal{Y}}(q^{*}\mathcal{L}'),
\]
 as follows directly from the fact that $p$ is an isomorphism between
Zariski open subsets of $\mathcal{X}'$ and $\mathcal{Y}$ and, as
a consequence, the $\Q-$line bundle $q^{*}\mathcal{L}'$ is trivial
on the support of the divisor $q^{*}\mathcal{K}_{\mathcal{X}}-\mathcal{K}_{\mathcal{Y}}.$

\emph{Step 2: $\mathcal{M}_{\mathcal{Y}}(p^{*}\mathcal{L}-E)=\mathcal{M}_{\mathcal{X}}(\mathcal{L})\implies p\text{ is an isomorphism, when \ensuremath{\mathcal{X}=\P_{\Z}^{1}}}$}

Replacing $q^{*}\mathcal{L}'$ with $p^{*}\mathcal{L}-E$ in formula
\ref{eq:difference of Mab} yields, since $\mathcal{M}_{\mathcal{Y}}(p^{*}\mathcal{L}-E)=\mathcal{M}_{\mathcal{X}}(\mathcal{L}),$

\[
(p^{*}\mathcal{L}-E)\cdot\left(\mathcal{K}_{\mathcal{Y}/\mathcal{X}}-\frac{1}{2}E\right)=0.
\]
It follows, since, by construction, $p^{*}\mathcal{L}-E$ is $p-$ample,
that
\begin{equation}
\mathcal{K}_{\mathcal{Y}/\mathcal{X}}=\frac{1}{2}E.\label{eq:pf equality case identity}
\end{equation}
Now, since $p:\,\mathcal{Y}\rightarrow\mathcal{X}$ is the blow-up
along a finite number of closed points $x_{i}$ in regular surfaces
$\mathcal{X}_{i},$ 
\begin{equation}
\mathcal{K}_{\mathcal{Y}/\mathcal{X}}=\sum c_{i}E_{i},\,\,\,c_{i}\geq1\label{eq:pf equality case ineq}
\end{equation}
 where the sum runs over all prime $p-$exceptional divisors $E_{i}.$
Hence, 
\[
E=\sum_{i}2c_{i}E_{i}\geq\sum_{i}2E_{i}.
\]
But this contradicts the isomorphisms \ref{eq:isom}, if the number
of points $x_{i}$ is non-zero. Indeed, denote by $E_{1}$ the strict
transform of the $p-$exceptional divisor on $\mathcal{Y}$ induced
from the exceptional divisor on the first point $x_{1}$ blown-up
on $\mathcal{X}(=\P_{\Z}^{1}).$ Then it it follows from the previous
inequality and the construction of $E$ that the restriction of the
ideal sheaf $\mathcal{J}$ on $\mathcal{X}$ to a neighbourhood of
$x_{1}$ in the fiber $\mathcal{X}_{\pi(x_{1})}$ is contained in
the $2k$th power $\ensuremath{\mathfrak{m}_{x_{1}}^{2k}}$ of the
maximal ideal $\mathfrak{m}_{x_{1}}$ on $\mathcal{X}_{\pi(x_{1})}$
defined by the point $x_{1}.$ But, in general, for $\mathcal{X}=\P_{\Z}^{1},$
the line bundle $k\mathcal{L}_{|\mathcal{X}_{\pi(x)}}\otimes\mathfrak{m}_{x}^{2k}$
on $\mathcal{X}_{\pi(x)}$ is trivial for any closed point $x$ on
$\mathcal{X}$ (since $\mathcal{L}_{|\mathcal{X}_{b}}:=-K_{\mathcal{X}_{b}}=\mathcal{O}_{\P_{\F_{b}}^{1}}(2)).$
But this contradicts the isomorphism \ref{eq:isom}, since $\mathcal{L}'$
is relatively ample. Hence, the number of points $x_{i}$ must be
zero, as desired. 

Combing these two steps thus concludes, as discussed above, the proof
of Theorem \ref{thm:Mab on log surface}. Finally, Corollary \ref{cor:Conjecture holds for log Fano surfaces}
can be deduced from Theorem \ref{thm:Mab on log surface} using a
generalization of Lemma \ref{lem:inf is inf}. But here we instead
proceeds as follows. Given an arithmetic log Fano surface $(\mathcal{X},\mathcal{D})$
set $\mathcal{L}:=-\mathcal{K}_{\mathcal{X}}$ and endow $\mathcal{L}$
and $-\mathcal{K}_{\mathcal{X}}$ with the same metric induced from
a volume-normalized metric on $-K_{X}$ with positive curvature current.
Then, by definition \ref{eq:more general log Mab},
\[
-\frac{1}{2}\overline{\mathcal{K}}_{(\mathcal{X},\mathcal{D})}^{2}=\mathcal{M}_{(\mathcal{X},\mathcal{D})}(\mathcal{L}).
\]
 Hence, combining Step one and Step two above yields 
\[
-\frac{1}{2}\overline{\mathcal{K}}_{(\mathcal{X},\mathcal{D})}^{2}\geq\mathcal{M}_{(\P_{\Z}^{1},0)}(\overline{-\mathcal{K}_{\P_{\Z}^{1}}})=-\frac{1}{2}(\overline{-\mathcal{K}_{\P_{\Z}^{1}}})^{2}
\]
 and the equality case is deduced precisely as before.
\begin{rem}
When $(\P^{1},\Delta)$ is K-stable equality holds in the inequality
\ref{eq:ineq log Mab in pf optimal} iff $(\mathcal{X}',\mathcal{D}')=(\P^{1},\mathcal{D}^{0}).$
Indeed, by Lemma \ref{lem:alpha geq one half}, the K-stability of
$(\P^{1},\Delta)$ implies that $\alpha(\P^{1},\mathcal{D}^{0})>1/2.$
Hence, if equality holds in \ref{eq:ineq log Mab in pf optimal},
then formula \ref{eq:difference Mab} forces $E=0,$ showing that
$p$ is an isomorphism. We can then conclude precisely as in the beginning
of Section \ref{subsec:The-equality-case}.
\end{rem}

\section{Appendix}

In the proof of Lemma \ref{lem:transf of height} we used the following
result (applied to $X=\P_{\C}^{n}$).
\begin{lem}
\label{lem:append}Let $X$ be a Fano manifold and $V$ a holomorphic
vector field on $X.$ Denote by $G_{t}$ the flow of the real part
of $V$ on $X$ at time $t$ and by $(G_{t}^{V})^{*}\phi$ its action
on a given continuous metric $\phi$ on $-K_{X}$ with positive curvature
current. If $X$ admits a Kähler-Einstein metric, then 
\[
\frac{d}{dt}\mathcal{E}(G_{t}^{*}\phi,\psi_{0})=0
\]
 for any fixed metric $\psi_{0}$ on $-K_{X}.$
\end{lem}

\begin{proof}
This is well-known and essentially goes back to \cite{fut}, but for
the convenience of the reader we provide a proof in the spirit of
the present paper and its precursor \cite{a-b}. Consider the \emph{Ding
functional} on the space of all continuous metrics on $-K_{X}$ with
positive curvature current, defined by 
\[
\mathcal{D}(\phi):=-\frac{n!}{(-K_{X})^{n}}\mathcal{E}(\phi)-\log\int_{X}\mu_{\phi},
\]
 where $\mu_{\phi}$ is the measure on $X$ corresponding to the metric
$\phi$ (see Section \ref{subsec:Log-Fano-varieties metrics}, with
$\Delta=0)$ and $\mathcal{E}(\phi)$ is a shorthand for the functional
$\mathcal{E}(\phi,\psi_{0})$ defined in formula \ref{eq:def of primitive}.
Since, $\mu_{G^{*}\phi}=G^{*}\mu_{\phi}$ for any biholomorphism $G$
of $X$ it follows that $\mathcal{E}(G_{t}^{*}\phi)$ and $\mathcal{D}(G_{t}{}^{*}\phi)$
have the same derivative. Moreover, in general, the function $t\mapsto\mathcal{D}(G_{t}^{*}\phi)$
is linear. Indeed, its derivative is the Futaki invariant of $V$
(see the claim on page 73 in \cite{ti}, where $\mathcal{D}$ is denoted
by $F_{\omega}$ ). Hence, all that remains is to verify that $\mathcal{D}$
is bounded from below (since then $t\mapsto\mathcal{D}(G_{t}^{*}\phi)$
must be constant). But this follows from the existence of a Kähler-Einstein
metric, since such a metric minimizes $\mathcal{D},$ as recalled
in \cite[Section 2.3]{a-b}. 
\end{proof}

\end{document}